\documentclass[11pt]{article}
\usepackage{enumerate}
\usepackage{amssymb,a4wide,latexsym,makeidx,epsfig,fleqn}
\usepackage{amsthm}
\usepackage{amsmath}
\usepackage{enumerate}
\usepackage{graphicx}
\usepackage{tikz}
\newtheorem{theorem}{Theorem}[section]
\newtheorem{remark}[theorem]{Remark}

\newtheorem{lemma}[theorem]{Lemma}

\begin{document}
\textwidth 150mm \textheight 225mm
\title{Improved bounds on the $H$-rank of a mixed graph in terms of the matching number and fractional matching number \footnote{Supported in part by the National Natural Science Foundation of China (Nos. 12371348,  12201258)}}

\author{Qi Wu$^{1}$\footnote{Email: wuqimath@163.com},  Yong Lu$^{2}$\footnote{Email: luyong@jsnu.edu.cn(Corresponding author)}
\\ {\small $^{1}$ School of Mathematical Sciences, Key Laboratory of MEA (Ministry of Education) \& Shanghai}
\\ {\small Key Laboratory of PMMP, East China Normal University, Shanghai, 200241, China}
\\{\small $^{2}$ School of Mathematics and Statistics, Jiangsu Normal University, Xuzhou,
Jiangsu 221116, China}}
\date{}
\maketitle
\begin{center}
\begin{minipage}{120mm}
\vskip 0.3cm
\begin{center}
{\small {\bf Abstract}}
\end{center}
{\small A mixed graph $\widetilde{G}$ is obtained by orienting some edges of a graph $G$, where $G$ is the underlying graph of $\widetilde{G}$. Let $r(\widetilde{G})$ be the $H$-rank of $\widetilde{G}$.
Denote by $r(G)$, $\kappa(G)$, $m(G)$ and $m^{\ast}(G)$ the rank, the  number of even cycles, the matching number and the fractional matching number of $G$, respectively.   Zhou et al. [Discrete Appl. Math. 313 (2022)] proved that $2m(G)-2\kappa(G)\leq r(G)\leq 2m(G)+\rho(G)$, where $\rho(G)$ is the largest number of disjoint odd cycles in $G$.
We extend their results to the setting of mixed graphs and prove that $2m(G)-2\kappa(G)\leq r(\widetilde{G}) \leq 2m^{\ast}(G)$ for a mixed graph $\widetilde{G}$. Furthermore,  we characterize some classes of mixed graphs with rank $r(\widetilde{G})=2m(G)-2\kappa(G)$, $r(\widetilde{G})=2m(G)-2\kappa(G)+1$ and $r(\widetilde{G})=2m^{\ast}(G)$, respectively. Our results also improve those of Chen et al. [Linear Multiliear Algebra. 66 (2018)].
In addition, our results can be applied to signed graphs and oriented  graphs in some situations.

 }
\end{minipage}
\end{center}

\section{Introduction }
%All graphs considered in this paper are finite and simple~(i.e., without multiedges and loops). Let $G$ be a  graph.  Denote by $V(G)$ and $E(G)$  the vertex
%set and the edge set of $G$, respectively.  Let $V(G)=\{v_{1},v_{2},\ldots,v_{n}\}$.  The \emph{adjacency matrix} $A(G)=(a_{ij})$ of $G$ is the symmetric $n\times n$ matrix with entries $a_{ij}=1$ if $v_{i}$ is adjacent to $v_{j}$ and $a_{ij}=0$ otherwise. The \emph{rank} (resp. \emph{nullity}) of $G$ is the rank (resp. nullity) of $A(G)$. Denote by  $r(G)$ (resp. $\eta(G)$) the rank (resp. nullity) of  $G$. Denote by $P_{n}$, $C_{n}$ and $K_{n}$ a \emph{path}, a \emph{cycle} and a \emph{complete graph} of order $n$, respectively.

In this paper, all graphs are finite and simple (no multiedges or loops). For a graph $G$ with vertex set $V(G) = \{v_1, \dots, v_n\}$ and edge set $E(G)$, its \textit{adjacency matrix} $A(G) = (a_{ij})_{n \times n}$ is defined as follows:
$a_{ij}=1$ if $v_{i}$ is adjacent to $v_{j}$ and $a_{ij}=0$ otherwise.
The \textit{rank} $r(G)$ and \textit{nullity} $\eta(G)$ of $G$ are the rank and nullity of $A(G)$, respectively. Standard notations $P_n$, $C_n$, and $K_n$ denote a \emph{path}, \emph{cycle}, and \emph{complete graph} of order $n$.

A graph $G$ is called \textit{singular} if its nullity $\eta(G)> 0$. The nullity of graphs plays an important role in the H$\mathrm{\ddot{u}}$ckel molecular orbital model. For a molecular graph $G$, if $\eta(G)>0$, then the corresponding chemical compound is highly reactive and unstable, or nonexistent (see \cite{AP, CIG}). In 1957, Collatz and Sinogowitz \cite{CS} first posed the challenging open problem of characterizing all singular graphs, which remains unsolved. There are lots of studies  on  the nullities (or ranks) of graphs (see \cite{BBD, CL,  LGUO, MWTDAM, SST, WW, ZWS, zwt}).

Let $G$ be a graph. The \emph{degree} of a vertex $x$ in $G$, denoted by $d_G(x)$, as the number of its neighbors in $G$. A vertex with degree $1$ is called a \emph{pendant vertex}, and its unique neighbor is referred to as a \emph{quasi-pendant vertex}.
A \emph{pendant cycle} of $G$ is a cycle which has unique vertex of degree $3$ and each other vertex has degree $2$ in $G$. A \emph{connected component} of $G$ is a maximal connected subgraph of $G$, and the number of such components is denoted by $\omega(G)$. The \emph{cyclomatic number} of  $G$,  denoted by $c(G)$, is defined as $c(G)=|E(G)|-|V (G)|+\omega(G)$. We call $G$  \emph{acyclic}, \emph{unicyclic} and \emph{bicyclic} if $c(G)=0$, $c(G)=1$ and $c(G)=2$, respectively. For a graph $G$ whose distinct cycles (if any) are pairwise vertex-disjoint,  we use $T_{G}$ to denote an acyclic
graph obtained from $G$ by contracting each cycle of $G$ into a vertex, called a \emph{cyclic vertex}.
The resulting acyclic graph after removing all cyclic vertices from $T_G$ is denoted by $[T_G]$.

A \emph{matching} of  $G$ is a set of  edges of $G$ such that any two edges in it are not incident. A matching with the maximum possible number of edges is a \emph{maximum matching} of $G$. The \emph{matching number} of  $G$, denoted by $m(G)$, is the size of a maximum matching of $G$. If all vertices of $G$ are endpoints of the edges of a matching of $G$, we call this matching  a \emph{perfect matching} of $G$. A \emph{fractional matching} of $G$ is defined as a function $f$ that assigns a weight to each edge within the closed interval $[0,1]$, satisfying $\sum_{e \in \Gamma(v)} f(e) \leq 1$ for every vertex $v \in V(G)$, where $\Gamma(v)$ represents the set of edges incident to $v$.
The \emph{fractional matching number} of $G$, denoted $m^{\ast}(G)$, is the maximum value of $\sum_{e \in E(G)} f(e)$ over all fractional matchings of $G$. It is straightforward to observe that $m^{\ast}(G) \geq m(G)$.
From the inequality $\sum_{e \in \Gamma(v)} f(e) \leq 1$ for each $v \in V(G)$, we deduce that $\sum_{e \in E(G)} f(e) = \frac{1}{2} \sum_{v \in V(G)} \sum_{e \in \Gamma(v)} f(e)$. Substituting the bound $\sum_{e \in \Gamma(v)} f(e) \leq 1$, this implies $\sum_{e \in E(G)} f(e) \leq \frac{1}{2} \sum_{v \in V(G)} 1 = \frac{n}{2}$ (with $n = |V(G)|$), so $m^{\ast}(G) \leq \frac{n}{2}$.
A \emph{fractional perfect matching} in $G$ is a fractional matching $f$ for which $\sum_{e \in E(G)} f(e) = \frac{n}{2}$; in this scenario, $m^{\ast}(G) = \frac{n}{2}$. If a fractional perfect matching takes only the values $0$ or $1$, it is equivalent to a (classical) perfect matching of $G$. Denote by $\kappa(G)$ the number of even cycles of graph $G$.

%A \emph{Sachs subgraph} is a mixed spanning subgraph with only $K_{2}$ or mixed cycles
%as components. Let $\widetilde{S}$ be a Sachs subgraph of $\widetilde{G}$. For some orientation of the cycles in $\widetilde{S}$,
%the signature of $\widetilde{S}$, denoted by $\sigma(S)$, is defined as $|f-b|$, where $f$ denotes the number of
%forward-oriented edges and $b$ denotes the number of backward-oriented edges of these
%mixed cycles in $\widetilde{S}$. A Sachs subgraph $\widetilde{S}$ is called \emph{basic} if $\sigma(\widetilde{C})$ is even for each cycle $\widetilde{C}$ in $\widetilde{S}$.

%In 2009, Guo et al. \cite{GYY} proved that if $G$ is a unicyclic graph, then $ r(G) = 2m(G) - 2, \, 2m(G), \text{ or } 2m(G) + 1. $
   %Wang and Wong \cite{WW} improved Guo's results: $ 2m(G) - 2c(G) \leq r(G) \leq 2m(G) + c(G). $
 %Feng et al. \cite{FHLL} improved Wang's results:
 % $ 2m(G) - 2c(G) \leq r(G) \leq 2m(G) + N_0, $
%  where $ N_0 $ is the number of odd cycles of $ G $. (When $ G $ is composed of the union of $ N_0 $ disjoint odd cycles).
%Ma and Fang \cite{MF} also improved Wang's results:
%  $ r(G) \leq 2m(G) + N_1, $
%  where $ N_1 $ is a nonnegative integer defined as: to make $ G $ to be a bipartite graph at least $ N_1 $ edges of $ G $ must be deleted from $ G $.
 %Zhou et al. \cite{ZWT} improved their results and proved that $2m(G)-2\kappa(G)\leq r(G)\leq 2m(G)+\rho(G)$, where $\rho(G)$ is the largest number of disjoint odd cycles in $G$.
%For fractional matching number, Chen and Guo \cite{CG0} proved that $2m^{\ast}(G)-2c(G)\leq r(G)\leq 2m^{\ast}(G)$.

 In 2009, Guo~et~al.~[12] established that for a unicyclic graph $G$, $r(G) \in \{2m(G)-2, 2m(G), 2m(G)+1\}$. Wang and Wong \cite{WW} extended Guo's results to arbitrary graphs $G$: $ 2m(G) - 2c(G) \leq r(G) \leq 2m(G) + c(G)$. Subsequent improvements to these bounds include:
\begin{itemize}
    %\item Wang--Wong~[24]: $2m(G) - 2c(G) \leq r(G) \leq 2m(G) + c(G)$
    \item Feng~et~al.~\cite{FHLL}: $2m(G) - 2c(G) \leq r(G) \leq 2m(G) + N_0$ where $N_0$ is the number of  odd cycles of $G$;
    \item Ma and Fang~\cite{MF}: $r(G) \leq 2m(G) + N_1$ with $N_1$ being the minimum edge-deletion size to make $G$  non-bipartite;
    \item Zhou~et~al.~\cite{ZWT}: $2m(G) - 2\kappa(G) \leq r(G) \leq 2m(G) + \rho(G)$ where $\rho(G)$ is the maximum number of disjoint odd cycles;
    \item Chen and Guo~\cite{CG0}: $2m^*(G) - 2c(G) \leq r(G) \leq 2m^*(G)$.
\end{itemize}

A \emph{mixed graph} $\widetilde{G}$ is formed by assigning directions to a subset of edges in an undirected graph $G$, which serves as the \emph{underlying graph} of $\widetilde{G}$. For such a mixed graph $\widetilde{G}$ of order $n$, its Hermitian adjacency matrix is a matrix $H(\widetilde{G}) = (h_{uv})_{n \times n}$, with entries defined as follows:
$h_{uv} = 1$ if $uv$ is an undirected edge;
$h_{uv} = i$ if there is an arc from $u$ to $v$;
$h_{uv} = -i$ if there is an arc from $v$ to $u$ and
$h_{uv} = 0$ otherwise.
The \textit{H-rank} of $\widetilde{G}$, denoted by $r(\widetilde{G})$, is the rank of  $H(\widetilde{G})$.

 Recent work has extended the study of  rank (or nullity) of graphs to mixed graphs. Mohar \cite{M} characterized all  mixed graphs with $H$-rank equal to 2. Wang et al. \cite{WYL} identified all mixed graphs with $H$-rank 3. Yang et al. \cite{YWY} gave some mixed graphs with $H$-rank 4, 6 or 8.  Chen et al. \cite{CLZ} established the relationship between the $H$-rank of a mixed graph and the rank of its underlying graph.  Li et al. \cite{LZX} derived bounds for the $H$-rank of a mixed graph using the independence number. Wei et al. \cite{WLM} studied the relationship between the $H$-rank of a mixed graph and the maximum degree of its underlying graph. The relationship between the $H$-rank of a mixed graph and the girth of its underlying graph was established by Khan \cite{K}.

A \emph{Sachs subgraph} of a mixed graph $\widetilde{G}$ is a mixed spanning subgraph with only $K_{2}$ or mixed cycles as components. For a given Sachs subgraph $\widetilde{S}$ and a chosen orientation of its cycles, the \emph{signature} of $\widetilde{S}$, denoted by $\sigma(\widetilde{S})$, is defined as $|f - b|$. Here, $f$ represents the number of forward-oriented edges and $b$ denotes the number of backward-oriented edges within the mixed cycles of $\widetilde{S}$. A Sachs subgraph $\widetilde{S}$ is termed \emph{basic} if for every cycle $C$ in $S$, the signature $\sigma(C)$ is an even integer.

Chen et al. \cite{CHL} studied the relation between the rank of  a mixed graph and the matching number of its underlying graph.
They obtained that if $\widetilde{G}$ is a mixed graph, then
\begin{equation}\label{eq:1}
2m(G)-2c(G)\leq r(\widetilde{G})\leq 2m(G)+c(G).
\end{equation}
Moreover, they characterized all mixed graphs $\widetilde{G}$ with rank $r(\widetilde{G})=2m(G)-2c(G)$.
\noindent\begin{theorem}\label{th:1.1}\cite{CHL}
Let $\widetilde{G}$ be a mixed graph. Then $r(\widetilde{G})=2m(G)-2c(G)$ if and only if the following conditions all hold:
\begin{enumerate}[(a)]
\item any two  cycles (if any) of $\widetilde{G}$ share no common vertices;
\item  each  cycle $\widetilde{C}_{q}$ of $\widetilde{G}$ satisfies  $\sigma(\widetilde{C}_{q})\equiv q~(\mathrm{mod}~4)$;
\item  $m(T_G)=m([T_G])$.
\end{enumerate}
\end{theorem}

 He et al. \cite{HHLG} and Chen and Guo \cite{CG1} proved that there is no mixed graph $\widetilde{G}$ with $H$-rank $r(\widetilde{G})=2m(G)-2c(G)+1$, respectively.

A \emph{signed graph} is a graph with a sign attached to each of its edges. Formally, a signed graph $\Gamma=(G, \sigma)$ consists of the underlying graph $G$ of $\Gamma$, and a sign function $\sigma: E\rightarrow\{+, -\}$. The \emph{adjacency matrix}  of $\Gamma$ is $A(\Gamma)=(a^{\sigma}_{ij})=\sigma(v_{i}v_{j})a_{ij}$, where $a_{ij}$ is an element of the adjacency matrix $A(G)$ of the underlying graph $G$. The \emph{rank} $r(\Gamma)$  of a signed graph $\Gamma$ is defined as the rank  of $A(\Gamma)$.

Let $\Gamma$ be a signed graph and $\widetilde{G}$ be a  mixed graph with the same underlying graph $G$. If $\widetilde{G}$ has  no directed edges and $\sigma=+$ (or $\sigma=-$) for $\Gamma$, then  $H(\widetilde{G})=A(\Gamma)=A(G)$ (or $H(\widetilde{G})=-A(\Gamma)=A(G)$). Hence, $H(\widetilde{G})$ and $A(\Gamma)$ have the same rank. If $\widetilde{G}$ has some directed edges, then by the definitions of $H(\widetilde{G})$ and $A(\Gamma)$,  there is no necessary connection between their ranks,  no matter whether $\sigma=+$,  $\sigma=-$, or $\sigma\in\{+, -\}$ in $\Gamma$. Hence the results about rank of signed graph and those of  mixed graph cannot be deduced from each other.

He et al.~\cite{HHL} gave bounds on the rank of a signed graph in terms of
matching number: $2m(G)-2c(G)\leq r(\Gamma)\leq 2m(G)+c(G)$. Chen and Guo \cite{CG} improved their results and obtained that $2m(G)-2\kappa(G)\leq r(\Gamma)\leq 2m^{\ast}(G)$.

Inspired by their results and building on it,   we improve the bounds of Inequality (1) and prove that $2m(G)-2\kappa(G)\leq r(\widetilde{G})\leq 2m^{\ast}(G)$ for a mixed graph $\widetilde{G}$. Furthermore, we characterize some classes of mixed graphs with rank $r(\widetilde{G})=2m(G)-2\kappa(G)$, $r(\widetilde{G})=2m(G)-2\kappa(G)+1$ and $r(\widetilde{G})=2m^{\ast}(G)$, respectively. Below are our main results.

\noindent\begin{theorem}\label{th:1.2}
Let $\widetilde{G}$ be a mixed graph. Then
\begin{equation}\label{eq:2}
2m(G)-2\kappa(G)\leq r(\widetilde{G})\leq 2m^{\ast}(G).
\end{equation}
\end{theorem}

When $0\leq \kappa(G)\leq c(G)$, $2m(G)-2c(G)\leq 2m(G)-2\kappa(G)$. According to Remark 2 and Lemma 3.1 in \cite{CG},  $2m^{\ast}(G)\leq 2m(G)+c(G)$ and  $2m^{\ast}(G)\leq 2m(G)+\rho(G)$. Thus the lower bound and upper bound of Inequality  (2) improve those of Inequality  (1) and the results of Zhou et al. \cite{ZWT}.

Let $C_{p}$ and $C_{q}$ be two vertex-disjoint cycles and $v\in V (C_{p}), u\in V (C_{q})$, $P_{l}=v_{1}v_{2}\ldots v_{l}~(l\geq 1)$ be a path of length $l-1$. Let $\infty(p, l, q)$ (as shown in Fig. 1) be
the graph obtained from $C_{p}, C_{q}$ and $P_{l}$ by identifying $v$ with $v_{1}$ and $u$ with $v_{l}$, respectively.  When $l=1$, the graph $\infty(p, 1, q)$ (as shown in Fig. 1) is obtained from $C_{p}$ and $C_{q}$ by identifying $v$ with $u$.

Let $P_{p}, P_{l}, P_{q}$ be three paths, where $\min\{p, l, q\}\geq 2$ and at most one of $p, l, q$
is $2$. Let $\theta(p, l, q)$ (as shown in Fig. 1) be the graph obtained from $P_{p}$, $P_{l}$ and $P_{q}$ by
identifying the three initial vertices and terminal vertices.

\begin{figure}[htbp]
\centering
 \includegraphics[scale=0.8]{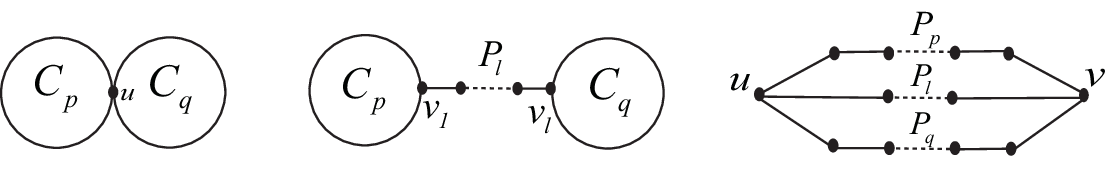}
\caption{$\infty(p, 1, q)$, $\infty(p, l, q)$ and $\theta(p, l, q)$.}
\end{figure}

Let $G$ be a graph and $S$ be a non-empty subset of its vertex set $V(G)$. The notation $G - S$ represents the induced subgraph formed by removing all vertices in $S$ and their incident edges from $G$. When $S$ consists of a single vertex $x$, we simplify the notation to $G - x$. Similarly, for any subgraph $H$ of $G$, we use $G - H$ to denote the subgraph obtained from $G$ by deleting all vertices of $H$ and their incident edges.
Denote by $H + x$  the subgraph induced by the vertex set $V(H) \cup \{x\}$.

Let $\mathfrak{G}_{0}$ be the set of connected mixed graphs $\widetilde{G}$~($c(G)\geq 1$) with pairwise vertex-disjoint cycles which satisfies the following properties: (a) $\widetilde{G}$ containing no pendant vertices; (b)  $\widetilde{G}$ has a pendant odd cycle $\widetilde{C}_{q}$ with odd signature; (c) $r(\widetilde{H})=2m(H)-2c(H)$~(i.e. $\widetilde{H}$ satisfies three conditions in Theorem \ref{th:1.1}), where $\widetilde{H}=\widetilde{G}-V(C_{q})+x$ and $x$ is the unique vertex of degree $3$ in $C_q$.

For an even integer $p$ and two odd integers $q$ and $l$, we define the following four sets:
\begin{itemize}
\item $\mathfrak{X}_{1}$: the set of mixed bicyclic graphs $\widetilde{\infty}(p,1,q)$ whose odd cycle  has odd signature and even cycle $\widetilde{C}_{p}$  satisfies $\sigma(\widetilde{C}_{p})\equiv p~(\mathrm{mod}~4)$;

\item  $\mathfrak{X}_{2}$: the set of mixed bicyclic graphs $\widetilde{\theta}(p,l,q)$ whose two odd cycles have odd signatures and  even cycle $\widetilde{C}_{q+l-2}$ satisfies $q+l+\sigma(\widetilde{C}_{q+l-2})\equiv 2~(\mathrm{mod}~ 4)$;

    \item $\mathfrak{Y}_{1}$: the set of mixed bicyclic graphs $\widetilde{\infty}(p,1,q)$ whose odd cycle  has even signature and even cycle $\widetilde{C}_{p}$  satisfies $\sigma(\widetilde{C}_{p})\equiv p~(\mathrm{mod}~4)$;

    \item $\mathfrak{Y}_{2}$: the set of mixed bicyclic graphs $\widetilde{\theta}(p,l,q)$ whose two odd cycles have even signatures and even cycle $\widetilde{C}_{q+l-2}$ satisfies $q+l+\sigma(\widetilde{C}_{q+l-2})\equiv 2~(\mathrm{mod}~ 4)$.
    \end{itemize}

Let $\mathfrak{G}_{1}=\mathfrak{G}_{0}\cup \mathfrak{X}_{1} \cup \mathfrak{X}_{2}$.  Let $\mathfrak{G}_{2}$ be the set consisting of $\mathfrak{G}_{1}$ and mixed odd cycles with odd signature.

If $u$ is a pendant vertex of a mixed graph $\widetilde{G}$ and $v$ is its unique neighbor in $\widetilde{G}$, then the
operation of obtaining $\widetilde{G}-\{u, v\}$ from $\widetilde{G}$ is called a \emph{pendant $K_{2}$ deletion}. The resultant subgraph $\widetilde{G}_{0}$ of $\widetilde{G}$ without pendant vertices, obtained from $\widetilde{G}$ by applying a series of pendant $K_2$ deletion operations, is called a \emph{crucial subgraph} of $\widetilde{G}$.

Let $\mathfrak{F}_{0}$ be the set of connected mixed graphs $\widetilde{G}$~($c(G)\geq 2$) with pairwise vertex-disjoint cycles which satisfies the following properties: (a) $\widetilde{G}$ containing no pendant vertices; (b)  $\widetilde{G}$ has a pendant odd cycle  with even signature and each other cycle  $\widetilde{C}_{p}$ is even cycle with $\sigma({\widetilde{C}_p})\equiv p~(\mathrm{mod}~4)$; (c) a series of pendant $K_2$ deletion operations can switch $\widetilde{H}$ to a crucial subgraph $\widetilde{H}_1$, which is disjoint union of $c(H)$ even cycles and an odd cycle with even signature, where $\widetilde{H}=\widetilde{G}-V(C_{q})+x$, $\widetilde{C}_{q}$ is a pendant even cycle of $\widetilde{G}$ and $x$ is the unique vertex of degree $3$ of $C_q$.

\begin{remark}
In the definition of $\mathfrak{F}_{0}$, since $\widetilde{G}$  is connected with pairwise vertex-disjoint cycles, $c(G)\geq 2$ and $\widetilde{G}$ contains no pendant vertices,  $\widetilde{G}$ has exactly two pendant cycles.
So $\widetilde{G}$ must have a pendant even cycle $\widetilde{C}_{q}$.
\end{remark}

Let $\mathfrak{F}_{1}=\mathfrak{F}_{0}\cup \mathfrak{Y}_{1}\cup \mathfrak{Y}_{2}$. Let $\mathfrak{F}_{2}$ be the set consisting of $\mathfrak{F}_{1}$ and mixed odd cycles with even signature.

We characterize some classes of mixed graphs with rank $r(\widetilde{G})=2m(G)-2\kappa(G)$ and $r(\widetilde{G})=2m(G)-2\kappa(G)+1$, respectively.

\noindent\begin{theorem}\label{th:1.3}
Let $\widetilde{G}$ be a mixed graph and $\kappa(G)=c(G)-1$. Then $r(\widetilde{G})=2m(G)-2\kappa(G)$ if and only if the following conditions both hold:
\begin{enumerate}[(a)]
\item any two even cycles (if any) of $\widetilde{G}$ share no common vertices;
%\item  each even cycle $\widetilde{C}_{q}$ (if any) of $\widetilde{G}$ satisfies  $\eta(\widetilde{C}_{q})\equiv q~(\mathrm{mod}~4)$;
\item  a series of pendant $K_{2}$ deletion operations can switch $\widetilde{G}$ to a crucial subgraph $\widetilde{G}_{1}$ with $\kappa(\widetilde{G}_{1})=\kappa(\widetilde{G})$, which is an empty  graph or has at most a connected component that is a graph in  $\mathfrak{G}_{2}$ and each other connected component is either an even cycle $\widetilde{C}_{q}$ with $\sigma(\widetilde{C}_{q})\equiv q~(\mathrm{mod}~4)$ or an isolated vertex.
\end{enumerate}
\end{theorem}

\noindent\begin{theorem}\label{th:1.4}
Let $\widetilde{G}$ be a mixed graph and $\kappa(G)=c(G)-1$. Then $r(\widetilde{G})=2m(G)-2\kappa(G)+1$ if and only if the following conditions both hold:
\begin{enumerate}[(a)]
\item any two even cycles (if any) of $\widetilde{G}$ share no common vertices;
%\item  each even cycle $\widetilde{C}_{q}$ (if any) of $\widetilde{G}$ satisfies $\eta(\widetilde{C}_{q})\equiv q~(\mathrm{mod}~4)$;
\item  a series of pendant $K_{2}$ deletion operations can switch $\widetilde{G}$ to a crucial subgraph $\widetilde{G}_{1}$ with $\kappa(\widetilde{G}_{1})=\kappa(\widetilde{G})$, which has exactly a connected component that is a graph in  $\mathfrak{F}_{2}$ and each other connected component is  either an even cycle $\widetilde{C}_{q}$ with $\sigma(\widetilde{C}_{q})\equiv q~(\mathrm{mod}~4)$ or an isolated vertex.
\end{enumerate}
\end{theorem}

Define $E_1(C)$ as the edge set where one endpoint resides in $C$ while the other lies outside $C$. Let $\mathcal{F}$ denote the collection of functions $f$ satisfying $\sum_{e \in E(G)} f(e) = m^{\ast}(G)$ and $f(e) \in \left\{ 0, \frac{1}{2}, 1 \right\}$ for all $e \in E(G)$. A fractional matching $f$ of graph $G$ is termed an optimal fractional matching if $f \in \mathcal{F}$ and $f$ attains the largest number of edges with $f(e)=1$ within $\mathcal{F}$. It is evident that a perfect matching is an optimal fractional matching. For an optimal fractional matching $f$, we define $E_2(G)$ as $\left\{ e \in E(G) \mid f(e) = 1 \right\}$.
\begin{remark}
According to \cite{SU}, for any graph $G$, $2m^{\ast}(G)$ is an integer. Moreover, there is a fractional matching $f$ for which
$$\sum_{e\in E(G)} f(e) = m^{\ast}(G),$$ such that $f(e)\in \{0, \frac{1}{2}, 1\}$ for every edge $e$. Thus for any graph $G$, $\mathcal{F}\neq \emptyset$.
\end{remark}

 Some class of mixed graphs with rank  $r(\widetilde{G})=2m^{\ast}(G)$ is identified as follows.
\noindent\begin{theorem}\label{th:1.5}
Let $\widetilde{G}$ be a bipartite cycle-disjoint mixed graph of order $n$. Then $\widetilde{G}$ is nonsingular if and only if the following conditions both hold:
\begin{enumerate}[(a)]
\item $G$ has a perfect fractional matching;
\item there exists an optimal fractional matching, for any even cycle $\widetilde{C}_{q}$ with $\sigma(\widetilde{C}_{q})\equiv q~(\mathrm{mod}~4)$, $E_{1}(C_{q})\cap E_{2}(G)\neq \emptyset$.
\end{enumerate}
\end{theorem}

\begin{remark}
In Theorem \ref{th:1.5}, if $\widetilde{G}$ is nonsingular, then $r(\widetilde{G})=n$. Recall that in definition of fractional matching, we show that $2m^{\ast}(G)\leq n$.
 By Inequality $(2)$, $r(\widetilde{G})\leq 2m^{\ast}(G)\leq n$.  Thus $r(\widetilde{G})=2m^{\ast}(G)= n$.
\end{remark}

An \emph{oriented graph} $G^{\sigma}$ obtained from $G$ is defined by assigning to each edge of $G$ a direction $\sigma$, where $G$ is called the underlying graph of $G^{\sigma}$. The \emph{skew-adjacency matrix} associated to $G^{\sigma}$ of order $n$, denoted by $S(G^{\sigma})$, is defined as a matrix $(s_{xy})_{n \times n}$ such that $s_{xy}=1$ if there is an arc from $x$ to $y$, $s_{xy }= -1$ if there is an arc from $y$ to $x$ and $s_{xy}=0$ otherwise. The \emph{skew-rank} of $G^{\sigma}$, denoted by $sr(G^{\sigma})$, is defined to be the rank of $sr(G^{\sigma})$. Obviously, an oriented graph $G^{\sigma}$ can be regraded as a mixed graph $\widetilde{G}$ without undirected edges. Consequently, $H(\widetilde{G})=iS(G^{\sigma})$ and $r(\widetilde{G})=sr(G^{\sigma})$. Hence Theorems \ref{th:1.2}, \ref{th:1.3}, \ref{th:1.4} and \ref{th:1.5} can be applied to  oriented graphs.

 %The rest of this paper is organized as follows. In Section 2 we review some known results. In Section 3  we give the proof of Theorem \ref{th:1.1}.
 %to characterize all $\mathbb{T}$-gain graphs $G^{\varphi}$ with rank $r(G^{\varphi})=r(G)-2c(G)+1$.
%In Section 4 we  first establish some technical lemmas and then prove Theorems \ref{th:1.3} and \ref{th:1.4}. In Section 5 we prove Theorem \ref{th:1.5}.

\section{Preliminaries}

~~~~In this section, we list some elementary lemmas and known results that will be
useful in the proof of our main results. For a mixed graph, we list the following lemmas.

%\begin{figure}[htbp]
% \centering
% \includegraphics[scale=0.8]{Fig.1.eps}
%\caption{$G$ and $G-v_5$.}
%\end{figure}

%For a mixed graph, we use the same notations as those
%of its underlying graph. In the following, we list some elementary lemmas and known results about mixed graphs.
\noindent\begin{lemma}\label{le:2.1}\cite{M}
If $v$ is a vertex of a mixed graph $\widetilde{G}$, then $r(\widetilde{G})-2\leq r(\widetilde{G}-v)\leq r(\widetilde{G})$.
\end{lemma}

\noindent\begin{lemma}\label{le:2.2}\cite{WYL}
Let $\widetilde{G}$ be a mixed graph containing a pendant vertex $x$ with the unique neighbor $y$. Then $r(\widetilde{G})=r(\widetilde{G}-x-y)+2$.
\end{lemma}

\noindent\begin{lemma}\label{le:2.3}\cite{M}
Let $\widetilde{G}$ be a mixed graph.
\begin{enumerate}[(a)]
\item If $\widetilde{G}=\widetilde{G}_{1}\cup \widetilde{G}_{2}\cup \cdots \cup \widetilde{G}_{t}$, where $\widetilde{G}_{1}, \widetilde{G}_{2},\ldots, \widetilde{G}_{t}$ are connected components of $\widetilde{G}$, then $r(\widetilde{G})=\Sigma_{i=1}^{t} r(\widetilde{G}_{i})$.
\item   $r(\widetilde{G})=0$ if and only if $\widetilde{G}$ is a graph without edges.
\end{enumerate}
\end{lemma}

 \noindent\begin{lemma}\label{le:2.4}\cite{CHL}
 Let $\widetilde{T}$ be a mixed tree. Then $r(\widetilde{T})=r(T)=2m(T)$.
\end{lemma}

\noindent\begin{lemma}\label{le:2.5}\cite{WYL}
%\cite{yqt}
Let $\widetilde{C}_{n}$ be a mixed cycle. Then
\begin{displaymath}
r(\widetilde{C_{n}})=\left\{\
\begin{array}{ll}
n-1, & \mathrm{if}~n~is~odd,~ \sigma(\widetilde{C}_{n})~is~odd;\\
n, & \mathrm{if}~n~is~odd,~ \sigma(\widetilde{C}_{n})~is~even;\\
n, & \mathrm{if}~n~is~even,~ \sigma(\widetilde{C}_{n})~is~odd;\\
n, & \mathrm{if}~n~is~even,~ n+\sigma(\widetilde{C}_{n})\equiv 2~(\mathrm{mod} ~4);\\
n-2, & \mathrm{if}~n~is~even,~ n+\sigma(\widetilde{C}_{n})\equiv 0~(\mathrm{mod} ~4).\\
\end{array}
\right.
\end{displaymath}
\end{lemma}

A \emph{gain graph} is a graph with the additional structure that each
orientation of an edge is given a group element, called a \emph{gain}, which is the inverse of
the group element assigned to the opposite orientation.  Denote by $\Phi=(G, \mathbb{T}, \varphi)$  a \emph{complex unit gain graph}, where $G$ is the underlying graph of $\Phi$, $\mathbb{T}= \{z\in \mathbb{C} : |z|=1\}$ is the \emph{circle group} and $\varphi: \overrightarrow{E} \rightarrow  \mathbb{T}$ is a gain function such that $\varphi(e_{ij})=\varphi(e_{ji})^{-1}=\overline{\varphi(e_{ji})}$.   The \emph{adjacency matrix} associated to  $\Phi$ is the Hermitian matrix $A(\Phi) = (a_{ij})_{n \times n}$, where $a_{ij}=\varphi(e_{ij})$ if $v_{i}$ is adjacent to $v_{j}$, otherwise $a_{ij}=0$.
%For convenience, $\varphi(e_{ij})$ is also written as  $\varphi(v_{i}v_{j})$ for $v_{i}v_{j}\in E(\Phi)$.
  Mixed graph is a special case when  $\varphi(\overrightarrow{E}) \subseteq \{1, i, -i\}$ for $\mathbb{T}$-gain graph. Hence, the following Lemma \ref{le:2.8} follows directly from Theorem 2.8 in \cite{LWX}.
\noindent\begin{lemma}\label{le:2.8}\cite{LWX}
Let $\widetilde{H}$ be a mixed graph and $\widetilde{C}_{n}$ be a mixed cycle.  Assume that $\widetilde{G}$ is the mixed graph obtained by identifying a vertex of $\widetilde{C}_{n}$ with a vertex of $\widetilde{H}$ (i.e. $V(\widetilde{C}_{n})\cap V(\widetilde{H})=x$). Let $\widetilde{F}=\widetilde{H}-x$. Then \begin{displaymath}
r(\widetilde{G})=\left\{\
\begin{array}{ll}
r(\widetilde{H})+n-1, & \mathrm{if}~n~is~odd,~ \sigma(\widetilde{C}_{n})~is~odd;\\
r(\widetilde{H})+n-2, & \mathrm{if}~n~is~even,~ \sigma(\widetilde{C}_{n})\equiv n~(\mathrm{mod}~4);\\
r(\widetilde{F})+n, & \mathrm{if}~n~is~even,~ either~ \sigma(\widetilde{C}_{n})+n\equiv 2~(\mathrm{mod}~4) ~or~ \sigma(\widetilde{C}_{n})~ is~ odd.
\end{array}
\right.
\end{displaymath}
\end{lemma}

\noindent\begin{lemma}\label{le:2.12}\cite{CG1,HHLG}
For any mixed graph $\widetilde{G}$, $r(\widetilde{G})\neq 2m(G)-2c(G)+1$.
\end{lemma}

Denote by  $\beta(G)$  the number of cycles in graph $G$.  The \emph{characteristic polynomial} of a mixed graph $\widetilde{G}$ of order $n$ is the characteristic polynomial of the Hermitian adjacency matrix of $\widetilde{G}$, defined as $f(\widetilde{G}, \lambda)=\mathrm{det}|\lambda I_{n}-H(\widetilde{G})|=\sum_{i=0}^{n}a_{i}\lambda^{n-i}$, where $I_{n}$ denotes the unite matrix of order $n$.
\noindent\begin{lemma}\label{le:2.9}\cite{CHL}
Let $\widetilde{G}$ be a mixed graph of order $n$. Then the characteristic polynomial $f(\widetilde{G}, \lambda)=\sum_{i=0}^{n}a_{i}\lambda^{n-i}$
has coefficients equal to $$a_{i}=\sum_{\widetilde{B}} (-1)^{\frac{1}{2}\sigma(\widetilde{B})+\omega(B)}2^{\beta(B)},$$
where the sum runs over all basic subgraphs $\widetilde{B}$ of order $i$ in $\widetilde{G}$.
\end{lemma}

\noindent\begin{lemma}\label{le:2.14}\cite{CHL}
Let $\widetilde{G}$ be a mixed graph containing a pendant vertex $x$ with unique neighbor $y$. If $r(\widetilde{G})= 2m(G)-2c(G)$, then
\begin{enumerate}[(a)]
\item $y$ lies outside any mixed cycle of $\widetilde{G}$;
\item $r(\widetilde{G}-x-y)=2m(G-x-y)-2c(G-x-y)$.
\end{enumerate}
\end{lemma}

For a simple graph, we have the following lemmas.
\noindent\begin{lemma}\label{le:2.6}\cite{CHL}
Let $G$ be a graph with a vertex $x$. Then $m(G)-1\leq m(G-x)\leq m(G)$.
\end{lemma}

\noindent\begin{lemma}\label{le:2.7}\cite{MWT}
Let $G$ be a graph with a pendant vertex $x$. Let $y$ be the unique neighbor of $x$. Then $m(G-y)=m(G-x-y)=m(G)-1$.
\end{lemma}

\noindent\begin{lemma}\label{le:2.10}\cite{CG}
Let $G$ be a graph with at least one cycle. Then the even cycles of $G$ are pairwise vertex-disjoint if and only if for any vertex $u$ of $G$ that lies on a even cycle,
$\kappa(G-u)=\kappa(G)-1$.
\end{lemma}

\noindent\begin{lemma}\label{le:2.11}\cite{MWT}
Let $T$ be a tree with at least one edge. If $r(T)=r(T-W)$ for a subset $W$ of $V(T)$, then there is a pendant vertex $v$ of $T$ such that $v\notin W$.
\end{lemma}

\noindent\begin{lemma}\label{le:2.13}\cite{SU}
Let $G$ be a bipartite graph. Then $m^{\ast}(G)=m(G)$.
\end{lemma}

\noindent\begin{lemma}\label{le:2.15}\cite{FHLL}
Let $G$ be a graph with  a pendant odd cycle $C_{q}$ and $x$ be a vertex of $C_{q}$  of degree $3$. Let $F=G-V(C_{q})$ and $H=F+x$. Then $m(G)=m(H)+\frac{q-1}{2}$.
\end{lemma}

Let $M$ be a matching of a graph $G$. We say that a vertex $v$ of $G$ is $M$-\emph{saturated} if it is an endpoint of some edge of $M$. Similarly, we say that a vertex $v$ of $G$ is $M$-\emph{unsaturated} if it is not any endpoint of any edge of $M$. An $M$-\emph{alternating path}
in $G$ is a path whose edges are alternately in $M$ and $E(G)\setminus M$. If neither its origin nor its
terminus is saturated by $M$, the path is called an $M$-\emph{augmenting path}.

\noindent\begin{lemma}\label{le:2.16}(Berge's Theorem)\cite{BM}
A matching $M$ of a graph $G$ is a maximum matching if and only if $G$ contains no $M$-augmenting path.
\end{lemma}

\section{Proof of Theorem \ref{th:1.2}}
~~~~In this section, we will  prove that  $2m(G)-2\kappa(G)\leq r(\widetilde{G})\leq 2m^{\ast}(G)$ for a mixed graph $\widetilde{G}$. At first, we need the following two lemmas.
%We first give several technical lemmas about the rank of a mixed graph, which is useful for the proof of Theorem \ref{th:1.1}.

\noindent\begin{lemma}\label{le:3.1}\cite{CG}
If an even-cycle-free graph has a perfect matching, then it has a unique maximum matching.
\end{lemma}

\noindent\begin{lemma}\label{le:3.2}
If a mixed graph $\widetilde{G}$ has a unique maximum matching, then $r(\widetilde{G})\geq 2m(G)$.
\end{lemma}
\begin{proof} Let $m= m(G)$ and $f(\widetilde{G}, \lambda)=\sum_{i=0}^{n}a_{i}\lambda^{n-i}$ be characteristic polynomial of $\widetilde{G}$.
By Lemma \ref{le:2.9}, $$a_{i}=\sum_{\widetilde{B}} (-1)^{\frac{1}{2}\sigma(\widetilde{B})+\omega(B)}2^{\beta(B)},$$
where the sum runs over all basic subgraphs $\widetilde{B}$ of order $i$ in $\widetilde{G}$.

 Since $\widetilde{G}$ has a unique maximum matching, $\widetilde{G}$ has a unique basic subgraph of order $2m$ with $m$
copies of $K_{2}$. As a result, the expression for $a_{2m}$ includes a summand $(-1)^{m}$. Apart from this summand, each other component in the expression must be an even
integer, since each takes  the form $(-1)^{\frac{1}{2}\sigma(\widetilde{B})+\omega(B)}2^{\beta(B)}$ with $\beta(B)\geq 1$. Therefore, $a_{2m}$ is an odd integer and  $a_{2m}\neq 0$. From this, it follows that $r(\widetilde{G})\geq 2m(G)$.
 \end{proof}

\noindent\textbf{Proof of Theorem \ref{th:1.2}:} We first prove the upper bound. Let $r=r(\widetilde{G})$ and $f(\widetilde{G}, \lambda)=\sum_{i=0}^{n}a_{i}\lambda^{n-i}$ be characteristic polynomial of $\widetilde{G}$.
By Lemma \ref{le:2.9}, $$a_{i}=\sum_{\widetilde{B}} (-1)^{\frac{1}{2}\sigma(\widetilde{B})+\omega(B)}2^{\beta(B)},$$
where the sum runs over all basic subgraphs $\widetilde{B}$ of order $i$ in $\widetilde{G}$. So $a_{r}\neq 0$ as $r=r(\widetilde{G})$. Then there is a basic subgraph $\widetilde{H}$ of order $r$. Since a basic subgraph consists of disjoint cycles and single edges, $r=2m^{\ast}(H)\leq 2m^{\ast}(G)$.

Now we prove the lower bound. We apply induction on $\kappa(G)$. Suppose that $\kappa(G)=0$. Let $M$ be a maximum matching of $G$. Let $\widetilde{F}$ be the subgraph of $\widetilde{G}$ induced by all endpoints of the edges  of $M$. Clearly, $F$ contains no even cycle and has a perfect matching. By Lemma \ref{le:3.1}, $F$ has a unique maximum matching. By Lemma \ref{le:3.2}, $r(\widetilde{F})\geq 2m(F)=2m(G)$. Since $\widetilde{F}$ is an induced subgraph of $\widetilde{G}$, $r(\widetilde{G})\geq r(\widetilde{F})\geq 2m(G)$.

 Now assume that $\kappa(G)\geq 1$. Let $u$ be a vertex on some even cycle of $G$. Obviously, $\kappa(G-u)\leq \kappa(G)-1$. By induction hypothesis, $r(\widetilde{G}-u)\geq 2m(G-u)-2\kappa(G-u)$. By Lemmas \ref{le:2.1} and \ref{le:2.6}, we have $r(\widetilde{G})\geq r(\widetilde{G}-u)\geq 2m(G-u)-2\kappa(G-u)\geq 2(m(G)-1)-2(\kappa(G)-1)=2m(G)-2\kappa(G)$.
 %There exists an edge $e$ of an even cycle such that it is not in a maximum matching of $G$. Let $H=G-e$. Then $m(H)=m(G)$. By induction hypothesis, $r(H)\geq 2m(H)-2\kappa(H)$.
 ~~~~~~~~~~~~~~~~~~~~~~~~~~~~~~~~~~~~~~~~~~~~~~~~~~~~~~~~~~~~~~~~~~~~~~~~~~~~~~~~~~~~~~~~~~~~~~~~~~~~\quad $\square$

\section{Proofs of Theorems \ref{th:1.3} and \ref{th:1.4}}

~\quad By the proof of Theorem \ref{th:1.2}, the above  inequalities  become equalities when  $r(\widetilde{G})= 2m(G)- 2\kappa(G)$. Hence we have the following Lemma \ref{le:4.1}.
\noindent\begin{lemma}\label{le:4.1}
Let $\widetilde{G}$ be a mixed graph and $r(\widetilde{G})= 2m(G)- 2\kappa(G)$, where
$0\leq \kappa(G)\leq c(G)$. Then
%\begin{enumerate}[(a)]
%\item   all even cycles are pairwise vertex-disjoint.
 for any vertex $u$ on any even cycle of $G$, $\kappa(G-u)=\kappa(G)-1$, $m(G-u)=m(G)-1$, $r(\widetilde{G})=r(\widetilde{G}-u)$ and $r(\widetilde{G}-u)= 2m(G-u)- 2\kappa(G-u)$.
%\end{enumerate}

\end{lemma}

\noindent\begin{lemma}\label{le:4.2}
Let $\widetilde{G}$ be a mixed graph and $r(\widetilde{G})= 2m(G)- 2\kappa(G)+1$, where
$0\leq \kappa(G)\leq c(G)$. Then
%\begin{enumerate}[(a)]
%\item   all even cycles are pairwise vertex-disjoint.
 for any vertex $u$ on any even cycle of $G$, $\kappa(G-u)=\kappa(G)-1$, $m(G-u)=m(G)-1$ and $r(\widetilde{G}-u)= 2m(G-u)- 2\kappa(G-u)+a_{u}$, where $a_{u}\in \{0,1\}$.
%\end{enumerate}

\end{lemma}

\begin{proof} Suppose on the contrary that $\kappa(G-u)\neq \kappa(G)-1$. So  $\kappa(G-u)\leq \kappa(G)-2$. By Lemmas \ref{le:2.1} and \ref{le:2.6} and  Inequality  (2),
$$\begin{aligned}
r(\widetilde{G})&\geq r(\widetilde{G}-u)\\
       &\geq 2m(G-u)-2\kappa(G-u)\\
       &\geq 2(m(G)-1)-2(\kappa(G)-2)\\
       &=2m(G)-2\kappa(G)+2,
\end{aligned}$$
a contradiction to $r(\widetilde{G}) = 2m(G)-2\kappa(G)+1$. Hence, $\kappa(G-u)=\kappa(G)-1$.

Suppose on the contrary that $m(G)= m(G-u)$.
By Lemma \ref{le:2.1} and Inequality (2),
$$\begin{aligned}
r(\widetilde{G})&\geq r(\widetilde{G}-u)\\
       &\geq 2m(G-u)-2\kappa(G-u)\\
       &=2m(G)-2(\kappa(G)-1)\\
       &=2m(G)-2\kappa(G)+2,
\end{aligned}$$
a contradiction to  $r(\widetilde{G})=2m(G)-2\kappa(G)+1$.
Hence, $m(G-u)=m(G)-1$.

Suppose on the contrary that $r(\widetilde{G}-u)\geq 2m(G-u)- 2\kappa(G-u)+2$. By Lemma \ref{le:2.1}, Inequality (2), $\kappa(G-u)=\kappa(G)-1$ and $m(G-u)=m(G)-1$,
$$\begin{aligned}
r(\widetilde{G})&\geq r(\widetilde{G}-u)\\
       &\geq 2m(G-u)-2\kappa(G-u)+2\\
       &= 2(m(G)-1)-2(\kappa(G)-1)+2\\
      &=2m(G)-2\kappa(G)+2,
\end{aligned}$$
a contradiction to  $r(\widetilde{G}) = 2m(G)-2\kappa(G)+1$. Hence, $r(\widetilde{G}-u)= 2m(G-u)- 2\kappa(G-u)+a_{u}$, where $a_{u}\in\{0,1\}$.
\end{proof}

\noindent\begin{lemma}\label{le:4.3}
Let $\widetilde{G}$ be a mixed graph  with a pendant vertex $u$, and $v$ be the unique neighbor of $u$ in $G$, $\widetilde{H}=\widetilde{G}-\{u,v\}$. If $r(\widetilde{G})= 2m(G)-2\kappa(G)+s$, where $s\in \{0,1\}$, then $v$ lies outside any even cycle of $G$.
\end{lemma}

\begin{proof} Suppose that the assertion is false. By Lemmas \ref{le:2.2} and \ref{le:2.7} and Inequality (2),
$$\begin{aligned}
   r(\widetilde{G})&=r(\widetilde{H})+2 \\
          &\geq 2m(H)-2\kappa(H)+2 \\
          &\geq 2m(G)-2-2(\kappa(G)-1)+2 \\
          &=2m(G)-2\kappa(G)+2,
 \end{aligned}$$ a contradiction. \end{proof}

 \noindent\begin{lemma}\label{le:4.4}
Let $\widetilde{G}$ be a mixed graph  with a pendant vertex $u$, and $v$ be the unique neighbor of $u$ in $G$,  $\widetilde{H}=\widetilde{G}-\{u,v\}$. If $v$ is not on any even cycle of $G$, then
$r(\widetilde{G})=2m(G)-2\kappa(G)+s$ if and only if $r(\widetilde{H})=2m(H)-2\kappa(H)+s$, where $0\leq \kappa(G)\leq c(G)$, $s$ is an integer and $0\leq s\leq 2m^{\ast}(G)-2m(G)+2\kappa(G)$.
\end{lemma}

\begin{proof}  Since $v$ is not on any even cycle of $G$, $\kappa(G)=\kappa(H)$.

\textbf{Necessity:} By Lemmas \ref{le:2.2} and \ref{le:2.7},
$$\begin{aligned}
   r(\widetilde{H})&=r(\widetilde{G})-2 \\
             &=2m(G)-2\kappa(G)+s-2 \\
             &=2m(H)-2\kappa(H)+s.
 \end{aligned}$$

\textbf{Sufficiency:} By Lemmas \ref{le:2.2} and \ref{le:2.7},
$$\begin{aligned}
   r(\widetilde{G})&=r(\widetilde{H})+2 \\
             &=2m(H)-2\kappa(H)+s+2 \\
             &=2m(G)-2\kappa(G)+s.
 \end{aligned}$$\end{proof}

 By Lemma \ref{le:2.3} and Inequality (\ref{eq:2}), we can obtain the following lemma.
 \noindent\begin{lemma}\label{le:4.5}
Let $\widetilde{G}$ be a  mixed graph. Then
\begin{enumerate}[(a)]
\item $r(\widetilde{G})=2m(G)-2\kappa(G)$ if and only if for each connected component $\widetilde{G}_{i}$ of $\widetilde{G}$, $r(\widetilde{G}_{i})=2m(G_{i})-2\kappa(G_{i})$.
\item $r(\widetilde{G})=2m(G)-2\kappa(G)+1$ if and only if $\widetilde{G}$ has exactly a connected component $\widetilde{G}_{0}$ satisfying that $r(\widetilde{G}_{0})=2m(G_{0})-2\kappa(G_{0})+1$ and for each  other connected component $\widetilde{G}_{i}$, $r(\widetilde{G}_{i})=2m(G_{i})-2\kappa(G_{i})$.
\end{enumerate}
\end{lemma}

 \noindent\begin{lemma}\label{le:4.6}
Let $\widetilde{G}$ be a  mixed graph  containing no pendant vertices and isolated vertices and $\kappa(G)=c(G)-1~(c(G)\geq 2)$.  If $r(\widetilde{G})=2m(G)-2\kappa(G)+s$, where $s\in \{0,1\}$, then  either $G$ has a perfect matching or $G$ has a connected component which is an odd cycle.
\end{lemma}

\begin{proof} Assume that $G$ is connected. Let $M$ be a maximum matching of $G$. Since $r(\widetilde{G})=2m(G)-2\kappa(G)+s$, $s\in \{0,1\}$, by Lemmas \ref{le:4.1} and \ref{le:4.2},   any vertex $u$ of any even cycle of $G$ is saturated by $M$ as $m(G-u)=m(G)-1$. Suppose that $M$  is not a perfect matching. Then there is an $M$-unsaturated vertex $v$ which is not on any even cycle of $G$. Let $w$ be a vertex on some even cycle of $G$. Then there is an $M$-alternating path between $v$ and $w$, say $P$. Otherwise, there is an $M$-augmenting path between $v$ and $w$, by Lemma \ref{le:2.16}, we can obtain a contradiction. Let $X$ be the set of edges of $P$ which are not in $M$ and $Y$ be the set of edges of $P$ which are in $M$. Let $M'=M\setminus Y \cup X$. Clearly, $|M'|=|M|$ and $w$ is an $M'$-unsaturated vertex, a contradiction to $m(G-w)=m(G)-1$. So $G$ has a perfect matching.

Now assume that $G$ is disconnected.
If $G$ has a connected component which is an odd cycle, then we are done. If not, each connected component $G_{j}$ of $G$ is an even cycle or $c(G_{j})\geq 2$.
Since $r(\widetilde{G})=2m(G)-2\kappa(G)+s$, $s\in\{0,1\}$, by Lemma \ref{le:4.5}, $\widetilde{G}$ has exactly a connected component $\widetilde{G}_{0}$ satisfying that $r(\widetilde{G}_{0})=2m(G_{0})-2\kappa(G_{0})+s$ and for each other  connected component $\widetilde{G}_{i}$, $r(\widetilde{G}_{i})=2m(G_{i})-2\kappa(G_{i})$.
For each connected component of $G$ which is an even cycle, it has a perfect matching. For each connected component $G_{t}$ of $G$ with $c(G_{t})\geq 2$,  similar as the above proof, $G_{t}$ has a perfect matching. Hence $G$ has a perfect matching.
\end{proof}

For mixed bicyclic graphs, we have the following two lemmas.

\noindent\begin{lemma}\label{le:4.7}
  Let $\widetilde{G}$ be a mixed bicyclic graph $\widetilde{\infty}(p,1,q)$, where $p$ is even and $q$ is odd. Then
   \begin{enumerate}[(a)]
  \item $r(\widetilde{G})=2m(G)-2\kappa(G)$ if and only if odd cycle of $\widetilde{G}$ has odd signature and even cycle $\widetilde{C}_{p}$ of it satisfies $\sigma(\widetilde{C}_{p})\equiv p~(\mathrm{mod}~4)$.
  \item $r(\widetilde{G})=2m(G)-2\kappa(G)+1$ if and only if odd cycle of $\widetilde{G}$ has even signature and even cycle $\widetilde{C}_{p}$ of it satisfies $\sigma(\widetilde{C}_{p})\equiv p~(\mathrm{mod}~4)$.
  \end{enumerate}
  \end{lemma}
\begin{proof}It is obvious that  $m(G)=(p+q-1)/2$ and $\kappa(G)=1$. We first prove (a) of this lemma.

\textbf{Sufficiency:} Since $\widetilde{C}_{q}$ has odd signature and $\widetilde{C}_{p}$  satisfies $\sigma(\widetilde{C}_{p})\equiv p~(\mathrm{mod}~4)$, by Lemmas \ref{le:2.5} and \ref{le:2.8}, $r(\widetilde{G})=p-2+r(\widetilde{C}_{q})=p+q-3=2m(G)-2\kappa(G)$.

\textbf{Necessity:} Assume that  $\widetilde{C}_{p}$ dose not satisfy $\sigma(\widetilde{C}_{p})\equiv p~(\mathrm{mod}~4)$. By Lemmas \ref{le:2.4} and \ref{le:2.8}, $r(\widetilde{G})=p+r(\widetilde{P}_{q-1})=p+q-1\neq 2m(G)-2\kappa(G)$, a contradiction. Hence $\sigma(\widetilde{C}_{p})\equiv p~(\mathrm{mod}~4)$. By Lemma \ref{le:2.8}, $r(\widetilde{G})=p-2+r(\widetilde{C}_{q})$. Since $r(\widetilde{G})=2m(G)-2\kappa(G)=p+q-3$, $r(\widetilde{C}_{q})=q-1$. By Lemma \ref{le:2.5}, $\widetilde{C}_q$ has odd signature.

Similarly, we can obtain (b) of this lemma.
\end{proof}

\noindent\begin{lemma}\label{le:4.8}
  Let $\widetilde{G}$ be a mixed bicyclic graph $\widetilde{\theta}(p,l,q)$, where $p$ is even,  $l$ and $q$ are odd. Then
  \begin{enumerate}[(a)]
  \item $r(\widetilde{G})=2m(G)-2\kappa(G)$ if and only if two  odd cycles of $\widetilde{G}$ have odd signatures and even cycle $\widetilde{C}_{q+l-2}$ satisfies $q+l+\sigma(\widetilde{C}_{q+l-2})\equiv 2~(\mathrm{mod}~ 4)$.
  \item $r(\widetilde{G})=2m(G)-2\kappa(G)+1$ if and only if two odd cycles of $\widetilde{G}$ have even signatures and even cycle $\widetilde{C}_{q+l-2}$ satisfies $q+l+\sigma(\widetilde{C}_{q+l-2})\equiv 2~(\mathrm{mod}~ 4)$.
  \end{enumerate}
\end{lemma}
\begin{proof} Let $m=m(G)$ and let  $f(\widetilde{G}, \lambda)=\sum_{i=0}^{n}a_{i}\lambda^{n-i}$ be characteristic polynomial of $\widetilde{G}$.
By Lemma \ref{le:2.9}, $$a_{i}=\sum_{\widetilde{B}} (-1)^{\frac{1}{2}\sigma(\widetilde{B})+\omega(B)}2^{\beta(B)},$$
where the sum runs over all basic subgraphs $\widetilde{B}$ of order $i$ in $\widetilde{G}$. Let $u$ be a vertex of even cycle $C_{q+l-2}$~(which is not on the odd cycle $C_{p+l-2}$) adjacent to vertex of degree 3 in $G$.
We first prove (a) of this lemma.

\textbf{Sufficiency:} It is obvious that $m(G)=(p+l+q-4)/2$ and $\kappa(G)=1$.  Since two  odd cycles of $\widetilde{G}$ have odd signatures, $\widetilde{G}$ has no basic subgraph of order $2m-1$. Hence $a_{2m-1}=0$. Note that there are three basic subgraphs of order $2m$ in $\widetilde{G}$. One of them is disjoint union of cycle $\widetilde{C}_{q+l-2}$ and a perfect matching of path $\widetilde{P}_{p-2}$. The other two are obtained by replacing cycle $\widetilde{C}_{q+l-2}$ with two distinct perfect matchings of it. So $a_{2m}= (-1)^{\frac{1}{2}\sigma(\widetilde{C}_{q+l-2})+\frac{p-2}{2}+1}\cdot 2+ (-1)^{\frac{p+l+q-4}{2}}\cdot 2=0$ as $q+l+\sigma(\widetilde{C}_{q+l-2})\equiv 2~(\mathrm{mod}~ 4)$. Hence $r(\widetilde{G})\leq 2m-2=2m(G)-2\kappa(G)$. By Inequality (2), we have that $r(\widetilde{G})=2m(G)-2\kappa(G)$.

\textbf{Necessity:} Without loss of generality, assume that odd cycle $\widetilde{C}_{p+l-2}$ has even signature. By Lemmas \ref{le:2.1}, \ref{le:2.2} and \ref{le:2.5}, $r(\widetilde{G})\geq r(\widetilde{G}-u)=q-3+r(\widetilde{C}_{p+l-2})=p+q+l-5=2m(G)-2\kappa(G)+1$, a contradiction. Hence two  odd cycles of $\widetilde{G}$ have odd signatures. So  the unique even cycle $\widetilde{C}_{q+l-2}$ of $\widetilde{G}$ has even signature. Then  $a_{2m}= (-1)^{\frac{1}{2}\sigma(\widetilde{C}_{q+l-2})+\frac{p-2}{2}+1}\cdot 2+ (-1)^{\frac{p+l+q-4}{2}}\cdot 2$. Since $r(\widetilde{G})=2m-2$, $a_{2m}=0$. Hence $q+l+\sigma(\widetilde{C}_{q+l-2})\equiv 2~(\mathrm{mod}~ 4)$.

Now we prove (b) of this lemma.

\textbf{Sufficiency:}  Since two  odd cycles of $\widetilde{G}$ have even signatures, by Lemmas \ref{le:2.1}, \ref{le:2.2} and \ref{le:2.5}, $r(\widetilde{G})\geq r(\widetilde{G}-u)=2m-1=2m(G)-2\kappa(G)+1$. Similar as the above proof, $a_{2m}=0$. So $r(G)=2m-1=2m(G)-2\kappa(G)+1$.

\textbf{Necessity:} Since $r(\widetilde{G})=2m(G)-2\kappa(G)+1=2m-1$, there is a basic subgraph of order $2m-1$. Hence at least one of two  odd cycles of $\widetilde{G}$ has even signature. Suppose that the other one has odd signature. Then the unique even cycle $\widetilde{C}_{q+l-2}$ has odd signature. So  $a_{2m}=(-1)^{\frac{p+l+q-4}{2}}\cdot 2\neq 0$. Hence $r(\widetilde{G})=2m$, a contradiction. So two odd cycles of $\widetilde{G}$ have even signatures. Then the unique even cycle $\widetilde{C}_{q+l-2}$ has even signature. So $a_{2m}= (-1)^{\frac{1}{2}\sigma(\widetilde{C}_{q+l-2})+\frac{p-2}{2}+1}\cdot 2+ (-1)^{\frac{p+l+q-4}{2}}\cdot 2=0$. Hence $q+l+\sigma(\widetilde{C}_{q+l-2})\equiv 2~(\mathrm{mod}~ 4)$.
\end{proof}

For a connected mixed graph, we have the following three lemmas.

\noindent\begin{lemma}\label{le:4.9}
Let $\widetilde{G}$ be a connected mixed graph  with $\kappa(G)=c(G)-1~(c(G)\geq 2)$ vertex-disjoint even cycles and $\widetilde{G}$ has no pendant vertices and no pendant even cycles.
Let $p$ be even and let $q$ and $l$ be odd.
 \begin{enumerate}[(a)]
  \item If $r(\widetilde{G})=2m(G)-2\kappa(G)$, then $\widetilde{G}$ is either a mixed bicyclic graph $\widetilde{\infty}(p,1,q)$ whose odd cycle  has odd signature and even cycle $\widetilde{C}_{p}$  satisfies $\sigma(\widetilde{C}_{p})\equiv p~(\mathrm{mod}~4)$ or a mixed bicyclic graph $\widetilde{\theta}(p,l,q)$ whose two odd cycles have odd signatures and even cycle $\widetilde{C}_{q+l-2}$ satisfies $q+l+\sigma(\widetilde{C}_{q+l-2})\equiv 2~(\mathrm{mod}~ 4)$.
  \item If $r(\widetilde{G})=2m(G)-2\kappa(G)+1$, then $\widetilde{G}$ is either a mixed bicyclic graph $\widetilde{\infty}(p,1,q)$ whose odd cycle  has even signature and even cycle $\widetilde{C}_{p}$  satisfies $\sigma(\widetilde{C}_{p})\equiv p~(\mathrm{mod}~4)$ or a mixed bicyclic graph $\widetilde{\theta}(p,l,q)$ whose two odd cycles have even signatures and even cycle $\widetilde{C}_{q+l-2}$ satisfies $q+l+\sigma(\widetilde{C}_{q+l-2})\equiv 2~(\mathrm{mod}~ 4)$.
  \end{enumerate}
\end{lemma}
\begin{proof}Let $\widetilde{G}$ be a connected mixed graph satisfying conditions of this lemma with rank $r(\widetilde{G})=2m(G)-2\kappa(G)+s$, where $s\in \{0,1\}$. Since $G$ has no pendant vertices and no pendant even cycles and $\kappa(G)=c(G)-1~(c(G)\geq 2)$,  any even cycle $\widetilde{C}^{i}$~($1\leq i\leq \kappa(G)$) in $\widetilde{G}$ must contain exactly two vertices of degree $3$ or one vertex of degree $4$. Denoted such vertices as $u_{i_{1}}$ and $u_{i_{2}}$~(if $u_{i_{1}}= u_{i_{2}}$, then this vertex has degree 4). Let $P_{i_{1}}$ and $P_{i_{2}}$ represent the two paths between $u_{i_{1}}$ and $u_{i_{2}}$ on $\widetilde{C}^{i}$. Since $r(\widetilde{G})=2m(G)-2\kappa(G)+s$, $s\in\{0,1\}$ and $G$ is connected with $c(G)\geq 2$, Lemma \ref{le:4.6} implies that $\widetilde{G}$ has a perfect matching. Consequently, $\widetilde{G}$ must contain at least an even cycle $\widetilde{C}^{i}$ such that both paths $P_{i_{1}}$ and  $P_{i_{2}}$ have odd length (without loss of generality, let this even cycle be $\widetilde{C}^{1}$). If, instead, any even cycle $\widetilde{C}^{i}$ has even paths $P_{i_{1}}$ and $P_{i_{2}}$,  then the existence of  a perfect matching would force  $\widetilde{G}$ to lack any odd cycle, a contradiction to $\kappa(G)=c(G)-1$.

Now, let $y$ be one of vertices adjacent to the vertex with degree 3 or 4 in $V(C^{1})$ (see an example shown in Fig. 2). By Lemmas \ref{le:4.1} and \ref{le:4.2},  $r(\widetilde{G}-y)=2m(G-y)-2\kappa(G-y)+t$, $t\in\{0,1\}$~(when $s=0$, $t=0$ and when $s=1$, $t\in\{0,1\}$). After pendant $K_{2}$ deletion operations for $\widetilde{G}-y$, we can get a crucial  subgraph $\widetilde{G}_{1}$. By Lemmas \ref{le:4.3} and \ref{le:4.4},  $r(\widetilde{G}_{1})=2m(G_{1})-2\kappa(G_{1})+t$, $t\in\{0,1\}$. If
$c(G_{1})\geq 2$, by Lemma \ref{le:4.6},  both of $\widetilde{G}$ and $\widetilde{G}_{1}$ have a perfect matching,
however, the order of $\widetilde{G}_{1}$ is odd, a contradiction. If $c(G_{1})=1$, then $\widetilde{G}_{1}$ is a mixed
odd cycle. So $\widetilde{G}$ is a mixed bicyclic graph $\widetilde{\infty}(p,1,q)$ or $\widetilde{\theta}(p,l,q)$. By Lemmas \ref{le:4.7} and \ref{le:4.8},  assertions (a) and (b) hold.
\end{proof}

\begin{figure}[htbp]
\centering
 \includegraphics[scale=1.5]{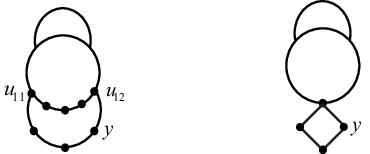}
\caption{$\widetilde{G}$.}
\end{figure}

\noindent\begin{lemma}\label{le:4.10}
Let $\widetilde{G}$ be a connected mixed graph  with pairwise vertex-disjoint  cycles and $\kappa(G)=c(G)-1~(c(G)\geq 2)$. If $\widetilde{G}$ has a pendant odd cycle $\widetilde{C}_{q}$ with odd signature and  $\widetilde{G}$ has no pendant vertices, then $r(\widetilde{G})= 2m(G)-2\kappa(G)$
if and only if $r(\widetilde{H})=2m(H)-2c(H)$, where $\widetilde{H}=\widetilde{G}-V(C_{q})+x$ and $x$ is the unique vertex of degree $3$ in $C_q$.
\end{lemma}
\begin{proof} Since $\widetilde{C}_{q}$ is a pendant odd cycle  with odd signature, by Lemma \ref{le:2.8}, $r(\widetilde{G})= q-1+r(\widetilde{H})$. By Lemma \ref{le:2.15}, $m(G)=m(H)+(q-1)/2$. Note that $\kappa(G)=\kappa(H)=c(H)$.
So $2m(G)-2\kappa(G)=2m(H)+q-1-2c(H)$. Combining this with $r(\widetilde{G})= q-1+r(\widetilde{H})$, we have $r(\widetilde{G})= 2m(G)-2\kappa(G)$
if and only if $r(\widetilde{H})=2m(H)-2c(H)$.
%\textbf{Necessity:}   By $r(\widetilde{G})=2m(G)-2\kappa(G)$, $r(\widetilde{G})=2m(G)-2\kappa(G)=2m(H)+q-1-2c(H)$. Combining this with $r(\widetilde{G})= q-1+r(\widetilde{H})$, we have  $r(\widetilde{H})=2m(H)-2c(H)$.
%\textbf{Sufficiency:}  By $r(\widetilde{H})=2m(H)-2c(H)$, $r(\widetilde{G})= q-1+r(\widetilde{H})=q-1+2m(H)-2c(H)=2m(G)-2\kappa(G)$.
\end{proof}

\noindent\begin{lemma}\label{le:4.11}
Let $\widetilde{G}$ be a connected mixed graph  with $\kappa(G)=c(G)-1~(c(G)\geq 2)$ vertex-disjoint even cycles and $\widetilde{G}$ has no pendant vertices. If
 $r(\widetilde{G})= 2m(G)-2\kappa(G)$, then $\widetilde{G}$ is a connected mixed graph in $\mathfrak{G}_{1}$.
\end{lemma}
\begin{proof}
 If $\widetilde{G}$ is  a connected mixed bicyclic graph in $\mathfrak{G}_{1}\setminus \mathfrak{G}_{0}$, then we are done. If not, since $r(\widetilde{G})= 2m(G)-2\kappa(G)$, by Lemma \ref{le:4.9}(a), $G$ has a pendant even cycle, say $C$. Let $x$ be a vertex in $V(C)$ with $d_{G}(x)=3$ and $y$ be one of vertices adjacent to $x$ in $V(C)$. By Lemma \ref{le:4.1}, $r(\widetilde{G}-y)=2m(G-y)-2\kappa(G-y)$.
After a series of pendant $K_{2}$ deletion operations for $\widetilde{G}-y$, we can get a crucial subgraph $\widetilde{F}$. By Lemmas \ref{le:4.3} and \ref{le:4.4}, $r(\widetilde{F})=2m(F)-2\kappa(F)$. By Lemma \ref{le:4.6}, $\widetilde{G}$ has a perfect matching. Hence the order of $\widetilde{F}$ is odd.

Suppose that $c(F)\geq 2$ and $\widetilde{F}$ has no isolated vertices. If $\widetilde{F}$ has no connected component which is an odd cycle, by Lemma \ref{le:4.6}, $\widetilde{F}$ has a perfect matching, a contradiction to the order of $\widetilde{F}$ is odd.
Hence  $\widetilde{F}$ has a connected component which is an odd cycle, say $C_{q}$. Then $C_q$ is a pendant odd cycle of $G$. Since $r(\widetilde{F})=2m(F)-2\kappa(F)$, by Lemma \ref{le:4.5}(a), $r(\widetilde{C}_{q})=2m(C_q)-2\kappa(C_q)=q-1$. By Lemma \ref{le:2.5}, $\widetilde{C}_{q}$ is an odd cycle with odd signature. Let $w$ be a vertex in $V(C_{q})$ with $d_{G}(w)=3$.
Let $\widetilde{H}=\widetilde{G}-V(C_q)+w$. By Lemma \ref{le:4.10}, $r(\widetilde{H})=2m(H)-2c(H)$. Thus $\widetilde{G}$ is a mixed graph in $\mathfrak{G}_{0}$, as required.

Otherwise, either $\widetilde{F}$ is isomorphic to a mixed odd  cycle (in this case, $\widetilde{G}$ is a mixed graph in $\mathfrak{G}_{0}$, as required), or $\widetilde{F}$ has at least one isolated vertex. Since $\widetilde{G}-y$ is a connected graph with exactly one pendant vertex,  by Lemma \ref{le:4.3}, the isolated vertex in $\widetilde{F}$ must lie  on the odd cycle of $\widetilde{G}-y$, this is impossible.
\end{proof}

For a mixed unicyclic graph, we have the following lemma.

\noindent\begin{lemma}\label{le:4.12}
Let $\widetilde{G}$ be a  mixed unicyclic graph  with unique odd cycle $\widetilde{C}_{q}$.
\begin{enumerate}[(a)]
\item If $r(\widetilde{G})= 2m(G)$, then $\widetilde{G}$ satisfies (b) of Theorem \ref{th:1.3}.
\item If $r(\widetilde{G})= 2m(G)+1$, then $\widetilde{G}$ satisfies (b) of Theorem \ref{th:1.4}.
\end{enumerate}
\end{lemma}
\begin{proof} We first prove (a) of this lemma. Suppose that the assertion is false. Let $\widetilde{G}_{0}$ be a counterexample with minimum order. Clearly, $r(\widetilde{G}_{0})= 2m(G_{0})$ and $\widetilde{G}_{0}$ does not satisfy (b) of Theorem \ref{th:1.3}.
If $G_{0}$ has a pendant vertex $x$ with unique neighbor $y$, then by Lemma \ref{le:4.4}, $r(\widetilde{G}_{0}-x-y)= 2m(G_{0}-x-y)$, a contradiction to minimality of order of $\widetilde{G}_{0}$. So $G_{0}$ consists of an odd cycle $C_q$ and some isolated vertices. Since $\widetilde{G}_{0}$ does not satisfy (b) of Theorem \ref{th:1.3}, $\widetilde{C}_q$ is a mixed odd cycle with even signature. So by Lemma \ref{le:2.5}, $r(\widetilde{G}_{0})=2m(G_{0})+1$, a contradiction.

Similarly, we can obtain (b) of this lemma.
\end{proof}

For lower bound of Inequality (1), we have the following lemma.

\noindent\begin{lemma}\label{le:4.13}
Let $\widetilde{G}$ be a mixed graph. If $r(\widetilde{G})=2m(G)-2c(G)$, then  a series of pendant $K_{2}$ deletion operations can switch $\widetilde{G}$ to a crucial subgraph $\widetilde{G}_{1}$, which is disjoint union of $c(G)$ even cycles and some isolated vertices~(possibly non-existent).
\end{lemma}
\begin{proof} %\textbf{Sufficiency:} By conditions (a)-(c) and Lemma \ref{le:2.5}, we have $r(\widetilde{G}_{1})=2m(G_1)-2c(G_1)$. Assume that switching $\widetilde{G}$ to $\widetilde{G}_{1}$ by $k$-step pendant $K_2$ deletion operations. By Lemmas \ref{le:2.2} and \ref{le:2.7}, $r(\widetilde{G})=r(\widetilde{G}_1)+2k=2m(G_1)-2c(G_1)+2k=2m(G)-2c(G)$.
%\textbf{Necessity:} By Theorem \ref{th:1.1}, we have (a) and (b) hold. Now we prove (c).
We apply induction on $|V(G)|$. If $|V(G)|\leq 2$, the assertion  holds trivially. Now assume that $|V(G)|\geq 3$. By Theorem \ref{th:1.1}(a) and (b), any two  cycles (if any) of $\widetilde{G}$ share no common vertices and
  each  cycle $\widetilde{C}_{q}$ of $\widetilde{G}$ satisfies  $\sigma(\widetilde{C}_{q})\equiv q~(\mathrm{mod}~4)$. If $T_G$ is an empty graph, then $G$ consists of disjoint cycles and some isolated vertices. So the assertion  holds. Suppose that $T_G$ has at least an edge. By Theorem \ref{th:1.1}(c), $m(T_G)=m([T_G])$. According to Lemma \ref{le:2.4}, $r(T_G)=r([T_G])$. By Lemma \ref{le:2.11}, $T_G$ has at least a pendant vertex which is not cyclic vertex, say $x$. Then $x$ is a pendant vertex of $G$. Let $y$ be the unique neighbor of $x$ in $G$.  By Lemma \ref{le:2.14}, $y$ lies outside any mixed cycle of $\widetilde{G}$ and
 $r(\widetilde{G}-x-y)=2m(G-x-y)-2c(G-x-y)$. By induction hypothesis,  a series of pendant $K_{2}$ deletion operations can switch $\widetilde{G}-x-y$ to a crucial subgraph $\widetilde{G}_{1}$, which is disjoint union of $c(G-x-y)$ even cycle and some isolated vertices~(possibly non-existent). Note that $c(G)=c(G-x-y)$. Thus  a series of pendant $K_{2}$ deletion operations can switch $\widetilde{G}$ to a crucial subgraph $\widetilde{G}_{1}$, which is disjoint union of $c(G)$ even cycle and some isolated vertices~(possibly non-existent).
\end{proof}

Now, we will prove Theorem \ref{th:1.3}.

\medskip
\noindent\textbf{Proof of Theorem \ref{th:1.3}.} \textbf{Sufficiency:} By conditions (a) and (b) and Lemmas \ref{le:2.5}, \ref{le:4.7}(a), \ref{le:4.8}(a) and \ref{le:4.10}, for each connected component $\widetilde{F}_{i}$ of $\widetilde{G}_{1}$, $r(\widetilde{F}_{i})=2m(F_{i})-2\kappa(F_{i})$. By Lemma \ref{le:4.5}(a), $r(\widetilde{G}_{1})=2m(G_1)-2\kappa(G_1)$. Suppose that switching $\widetilde{G}$ to $\widetilde{G}_{1}$ by $k$-step pendant $K_{2}$ deletion operations. By Lemmas \ref{le:2.2} and \ref{le:2.7}, $r(G)=r(G_{1})+2k=2m(G_1)-2\kappa(G_1)+2k=2m(G)-2\kappa(G)$.

\textbf{Necessity:}
 By Lemma \ref{le:4.5}(a), we can assume that $\widetilde{G}$ is connected. Since $r(\widetilde{G})=2m(G)-2\kappa(G)$, by Lemma \ref{le:4.1}, $\kappa(G-u)=\kappa(G)-1$ for any vertex $u$ on any even cycle of $G$. By Lemma \ref{le:2.10}, assertion (a) holds.

 For assertion (b), we apply induction on the order of $G$. If $|V(G)|\leq 2$, then the results hold trivially.
Now assume that $|V(G)|\geq 3$. If $c(G)=0$, then $G$ is an acyclic graph. By Lemma \ref{le:2.4},  $r(\widetilde{G})=2m(G)$ and we are done.
If $c(G)=1$, then $\kappa(G)=c(G)-1=0$ and $\widetilde{G}$ is a mixed unicyclic  graph with an odd cycle.  By Lemma \ref{le:4.12}(a),  assertion (b) holds.

Suppose that $c(G)\geq 2$. Note that $\mathfrak{G}_{1}\subseteq \mathfrak{G}_{2}$. If $\widetilde{G}$ is a mixed graph in $\mathfrak{G}_{1}$, then we are done.  If not, by Lemma \ref{le:4.11}, $\widetilde{G}$ has at least one pendant vertex, say $x$. Let $y$ be the unique neighbor of $x$ in $G$.
 Let $\widetilde{H}=\widetilde{G}-\{x,y\}$. If  $y$ is on the odd cycle of $G$, then by Lemma \ref{le:4.4},
$r(\widetilde{H})= m(H)-2\kappa(H)= m(H)-2c(H)$. By Theorem \ref{th:1.1} (a) and (b), any two  cycles (if any) of $\widetilde{H}$ share no common vertices and
  each  cycle $\widetilde{C}_{q}$ of $\widetilde{H}$ satisfies  $\sigma(\widetilde{C}_{q})\equiv q~(\mathrm{mod}~4)$. Then by Lemma \ref{le:4.13},  assertion (b) holds.
Suppose that $y$ is not on the odd cycle of $G$. Since $r(\widetilde{G})=m(G)-2\kappa(G)$, by Lemma \ref{le:4.3}, $y$ is not on any even cycle of $G$.
By Lemma \ref{le:4.4}, $r(\widetilde{H})=m(H)-2\kappa(H)$.
Note that $\kappa(G)= \kappa(H)= c(H)-1=c(G)-1$. By induction hypothesis, $\widetilde{H}$ satisfies condition (b). Thus $\widetilde{G}$ satisfies condition (b). ~~~~~~~~~~~~~~~~~~~~~~~~~~~~~~\quad $\square$
%~~~~~~~~~~~~~~~~~~~~~~~~~~~~~~~~~~~~~~~~~~~~~~~~~~~~~~~~~~~~~~~~~~~~~~~~~~~~\quad $\square$

\medskip
Now, we give an example that satisfies the equality in Theorem \ref{th:1.3}. Let $\widetilde{G}$ be a  mixed graph with an odd cycle $\widetilde{C}_{3}$ having odd signature and each other cycle  is $\widetilde{C}_{4}$ with $\sigma(\widetilde{C}_{4})\equiv 0~(\mathrm{mod}~ 4)$ shown in Fig. 3.
By Lemmas \ref{le:2.2}, \ref{le:2.5} and \ref{le:2.7}, $r(\widetilde{G})=r(\widetilde{G}-\{u,v\})+2=(4-2)\kappa(G)+3-1+2=2\kappa(G)+4$ and $m(G)=m(G-\{u,v\})+1=2\kappa(G)+1+1=2\kappa(G)+2$.  Thus $r(\widetilde{G})=2m(G)-2\kappa(G)$.

We give the following four lemmas which will be useful in the proof of Theorem \ref{th:1.4}.
\noindent\begin{lemma}\label{le:4.14}
Let $\widetilde{G}$ be a  mixed graph  with pairwise vertex-disjoint  cycles and $\widetilde{G}$ has no pendant vertices. If $\widetilde{G}$ has a pendant even cycle $\widetilde{C}_{q}$ with $\sigma(\widetilde{C}_{q})\equiv q~(\mathrm{mod}~ 4)$ and  $r(\widetilde{G})=2m(G)-2\kappa(G)+1$, then $r(\widetilde{H})=2m(H)-2\kappa(H)+1$, where  $\widetilde{H}=\widetilde{G}-V(C_q)+x$ and $x$ is the unique vertex of degree $3$ in $C_q$.
\end{lemma}
\begin{proof}
 Let  $y$ be a vertex adjacent to $x$ in $V(C_{q})$.
 Since $r(\widetilde{G})=2m(G)-2\kappa(G)+1$, by Lemma \ref{le:4.2},  $m(G-y)=m(G)-1$. By Lemma \ref{le:2.7}, $m(G-y)=m(H)+(q-2)/2$. Hence $m(H)=m(G)-q/2$.
Since $\widetilde{C}_{q}$ is an even cycle with $\sigma(\widetilde{C}_{q})\equiv q~(\mathrm{mod}~ 4)$, by Lemma \ref{le:2.8}, $r(\widetilde{H})+q-2=r(\widetilde{G})=2m(G)-2\kappa(G)+1=2m(H)+q-2\kappa(H)-2+1$. It follows that $r(\widetilde{H})=2m(H)-2\kappa(H)+1$.
\end{proof}

\begin{figure}[htbp]
\centering
 \includegraphics[scale=1.5]{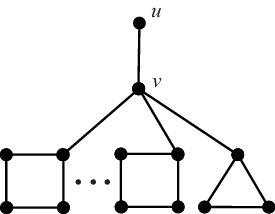}
\caption{An example that satisfies the equality in Theorem \ref{th:1.3} (or \ref{th:1.4}).}
\end{figure}

\noindent\begin{lemma}\label{le:4.15}
Let $\widetilde{G}$ be a  mixed graph  with pairwise vertex-disjoint cycles and $\widetilde{G}$ has   no odd cycles. Then $r(\widetilde{G})\neq 2m(G)-2\kappa(G)+1$.
\end{lemma}
\begin{proof} Suppose that  the assertion is false. Let $\widetilde{G}_{0}$ be a counterexample with minimum order. Clearly, $r(\widetilde{G}_{0})= 2m(G_{0})-2\kappa(G_{0})+1$. If $c(G_{0})=0$, then $G_{0}$ is an acyclic graph. By Lemma \ref{le:2.4}, $r(\widetilde{G}_{0})= 2m(G_{0})$, a contradiction. Suppose that $c(G_{0})\geq 1$. If $G_0$ has a pendant vertex $u$ with unique neighbor $v$, by Lemmas \ref{le:4.3} and \ref{le:4.4}, $r(\widetilde{G}_{0}-u-v)= 2m(G_{0}-u-v)-2\kappa(G_{0}-u-v)+1$, a contradiction to minimality of order of $\widetilde{G}_0$. Thus $G_0$ has no pendant vertices. Note that $G_{0}$ has no odd cycles. If $G_0$ has no pendant even cycles, then $G_0$ consists of disjoint even cycles and some isolated vertices. By Lemmas \ref{le:2.3} and \ref{le:2.5},
$r(\widetilde{G}_0)\neq 2m(G_{0})-2\kappa(G_{0})+1$, a contradiction.

Now assume that  $\widetilde{G}_{0}$ has a pendant even cycle, say $\widetilde{C}_{q}$. Let $x$ be the unique vertex of degree 3 in $C_q$. Let $y$ be a vertex of $C_q$ adjacent to $x$. Let $\widetilde{F}=\widetilde{G}_{0}-V(C_{q})$ and $\widetilde{H}=\widetilde{F}+x$. Since $r(\widetilde{G}_{0})= 2m(G_{0})-2\kappa(G_{0})+1$, by Lemmas \ref{le:2.7} and \ref{le:4.2}, $m(F)=m(G_{0}-x)-(q-2)/2=m(G_{0})-1-(q-2)/2=m(G_{0})-q/2$ and $m(H)=m(G_{0}-y)-(q-2)/2=m(G_{0})-1-(q-2)/2=m(G_{0})-q/2$.
If $\sigma(\widetilde{C}_{q})$  is odd or $\sigma(\widetilde{C}_{q})+q \equiv 2~(\mathrm{mod}~4)$, then by Lemma \ref{le:2.8}, $r(\widetilde{F})+q =r(\widetilde{G}_{0})=2m(G_{0})-2\kappa(G_{0})+1=2m(F)+q-2\kappa(F)-2+1$. Thus $r(\widetilde{F})=2m(F)-2\kappa(F)-1$, a contradiction to Inequality (2). So $\widetilde{C}_{q}$ satisfies $\sigma(\widetilde{C}_{q})\equiv q~(\mathrm{mod}~4)$.
By Lemma \ref{le:2.8}, $r(\widetilde{H})+q-2=r(\widetilde{G}_{0})=2m(G_{0})-2\kappa(G_{0})+1=2m(H)+q-2\kappa(H)-2+1$. It follows that $r(\widetilde{H})=2m(H)-2\kappa(H)+1$, a contradiction to minimality of order of $\widetilde{G}_0$.
\end{proof}

\noindent\begin{lemma}\label{le:4.16}
Let $\widetilde{G}$ be a connected mixed graph  with $\kappa(G)=c(G)-1~(c(G)\geq 2)$ vertex-disjoint even cycles and $\widetilde{G}$ has no pendant vertices. If $r(\widetilde{G})= 2m(G)-2\kappa(G)+1$, then $\widetilde{G}$ is a connected mixed graph in $\mathfrak{F}_{1}$.
\end{lemma}
\begin{proof} If $\widetilde{G}$ is a connected mixed bicyclic graph in $\mathfrak{F}_{1}\setminus \mathfrak{F}_{0}$, then we are done. If not, since $r(\widetilde{G})= 2m(G)-2\kappa(G)+1$, by Lemma \ref{le:4.9}(b), $\widetilde{G}$ has a pendant even cycle, say $\widetilde{C}_{q}$. Let $x$ be a vertex in $V(C_q)$ with $d_{G}(x)=3$ and $y$ be a vertex adjacent to $x$ in $V(C_q)$. Let $\widetilde{F}=\widetilde{G}-V(C_{q})$ and $\widetilde{H}=\widetilde{F}+x$. By Lemmas \ref{le:2.7} and \ref{le:4.2}, $m(F)=m(G-x)-(q-2)/2=m(G)-q/2$ and $m(H)=m(G-y)-(q-2)/2=m(G)-q/2$.

If $\sigma(\widetilde{C}_{q})$ is odd or $\sigma(\widetilde{C}_{q})+q \equiv 2~(\mathrm{mod}~4)$, then by Lemma \ref{le:2.8}, $r(\widetilde{G})=r(\widetilde{F})+q=2m(G)-2\kappa(G)+1=2m(F)+q-2\kappa(F)-2+1$. Thus $r(\widetilde{F})=2m(F)-2\kappa(F)-1$, a contradiction to Inequality (2). So $\widetilde{C}_{q}$ satisfies $\eta(\widetilde{C}_{q})\equiv q~(\mathrm{mod}~4)$. By Lemma \ref{le:2.8}, $r(\widetilde{G})=r(\widetilde{H})+q-2=2m(G)-2\kappa(G)+1=2m(H)+q-2\kappa(H)-2+1$. It follows that $r(\widetilde{H})=2m(H)-2\kappa(H)+1$. Applying pendant $K_2$ deletion operations on $\widetilde{H}$, we can obtain a crucial subgraph $\widetilde{H}_{1}$ such that $r(\widetilde{H}_{1})=2m(H_1)-2\kappa(H_1)+1$. Since $r(\widetilde{G})= 2m(G)-2\kappa(G)+1$, by Lemma \ref{le:4.6}, $\widetilde{G}$ has a perfect matching. So the order of $H_1$ is odd. By Lemma \ref{le:4.6}, $\widetilde{H}_1$ has a connected component which is an odd cycle, say $\widetilde{C}_{t}$. If $\widetilde{C}_{t}$ has odd signature, then by Lemma \ref{le:2.5}, $r(\widetilde{C}_{t})=2m(C_{t})$. So by Lemma \ref{le:4.5}(b), $\widetilde{H}_{1}$ has a connected component $\widetilde{H}_{2}$~(different from $\widetilde{C}_{t}$) satisfying that $r(\widetilde{H}_{2})=2m(H_{2})-2\kappa(H_{2})+1$. Note that $\widetilde{H}_{2}$ has no odd cycles.  By Lemma \ref{le:4.15}, we obtain a contradiction. So $\widetilde{C}_{t}$ has even signature. Then by Lemma \ref{le:2.5}, $r(\widetilde{C}_{t})=2m(C_{t})+1$. So by Lemma \ref{le:4.5}(b), for each other connected component $\widetilde{H}^{i}$ of $\widetilde{H}_{1}$, $r(\widetilde{H}^{i})=2m(H^{i})-2\kappa(H^{i})=2m(H^{i})-2c(H^{i})$. So each other connected component $\widetilde{H}^{i}$ of $\widetilde{H}_{1}$ satisfies assertions (a) and (b) of Theorem \ref{th:1.1} and the assertion of Lemma \ref{le:4.13}. Then $\widetilde{G}$ is a mixed graph in $\mathfrak{F}_{0}$, as required.
\end{proof}

\noindent\begin{lemma}\label{le:4.17}
Let $\widetilde{G}$ be a connected mixed graph in $\mathfrak{F}_{2}$. Then $r(\widetilde{G})= 2m(G)-2\kappa(G)+1$.
\end{lemma}
\begin{proof}If $\widetilde{G}$ is  a connected mixed graph in $\mathfrak{F}_{2}\setminus \mathfrak{F}_{0}$, then by Lemmas \ref{le:2.5}, \ref{le:4.7}(b) and \ref{le:4.8}(b), the result holds. Now assume that $\widetilde{G}$ is a mixed graph in $\mathfrak{F}_{0}$. In property (c) of definition of $\mathfrak{F}_{0}$, assume that switching $\widetilde{H}$ to $\widetilde{H}_{1}$ by $k$-step pendant $K_2$ deletion operations. By Lemma \ref{le:2.7}, $m(H)=m(H_1)+2k$. Let $S$ be the set of deleted vertices by switching $\widetilde{H}$ to $\widetilde{H}_{1}$. Let $F$ be the subgraph of $H$ induced by $S$. Let $M_1$ be a maximum matching of $H_1$.  We can select an appropriate matching $M_2$ of size $k$ from $E(F)$ such that $x$ is $M_2$-unsaturated and edges of $M_1$ and $M_2$ have no common incident vertices. Let $M=M_1\cup M_2$. So $M$ is a maximum matching of $H$. Note that $H=G-V(C_q)+x$. Then $m(G)=m(H)+q/2$ as $x$ is an $M$-unsaturated vertex. By properties (a)-(c) of definition of $\mathfrak{F}_{0}$ and Lemmas \ref{le:2.2}, \ref{le:2.5} and \ref{le:2.7}, $r(H)=2m(H)-2\kappa(H)+1$. Since $\widetilde{C}_{q}$ is an even cycle and $\sigma(\widetilde{C}_{q})\equiv q~(\mathrm{mod}~4)$, by Lemma \ref{le:2.8}, $r(\widetilde{G})=r(\widetilde{H})+q-2=2m(H)-2\kappa(H)+1+q-2=2m(G)-2\kappa(G)+1$.
\end{proof}

\medskip
We are ready to give a proof of Theorem \ref{th:1.4}.

\noindent\textbf{Proof of Theorem \ref{th:1.4}.} \textbf{Sufficiency:} By conditions (a) and (b) and Lemmas \ref{le:2.3}, \ref{le:2.5} and \ref{le:4.17}, $r(\widetilde{G}_{1})=2m(G_1)-2\kappa(G_1)+1$. Assume that switching $\widetilde{G}$ to $\widetilde{G}_{1}$ by $k$-step pendant $K_2$ deletion operations. So by Lemmas \ref{le:2.2} and \ref{le:2.7}, $r(\widetilde{G})=r(\widetilde{G}_{1})+2k=2m(G_1)-2\kappa(G_1)+2k=2m(G)-2\kappa(G)$.

\textbf{Necessity:}
 By Theorem \ref{th:1.3} and Lemma \ref{le:4.5}(b), we can assume that $\widetilde{G}$ is connected. Since $r(\widetilde{G})=2m(G)-2\kappa(G)+1$, by Lemma \ref{le:4.2}, $\kappa(G-u)=\kappa(G)-1$ for any vertex $u$ on any even cycle of $G$. By Lemma \ref{le:2.10},  assertion (a) holds.

 For assertion (b), we apply induction on the order of $G$.  If $c(G)=0$, then by Lemma \ref{le:2.4},  $r(\widetilde{G})=2m(G)$, a contradiction. Thus $c(G)\geq 1$. So  $|V(G)|\geq 3$.  If $|V(G)|= 3$, then  $\widetilde{G}$ is a mixed cycle $\widetilde{C}_{3}$. By Lemma \ref{le:2.5}, $\widetilde{G}$ is a mixed cycle $\widetilde{C}_{3}$ with even signature as $r(\widetilde{G})=2m(G)-2\kappa(G)+1$. Then $\widetilde{G}$ is a mixed graph in $\mathfrak{F}_{2}$. So assertion (b)  holds.
Now assume that $|V(G)|\geq 4$.
If $c(G)=1$, then $\kappa(G)=c(G)-1=0$ and $\widetilde{G}$ is a mixed unicyclic  graph with an
odd cycle.  Since $r(\widetilde{G})=2m(G)-2\kappa(G)+1=2m(G)+1$, by Lemma \ref{le:4.12}(b),  assertion (b) holds.

Suppose that $c(G)\geq 2$. Note that $\mathfrak{F}_{1}\subseteq \mathfrak{F}_{2}$. If $\widetilde{G}$ is a mixed graph in $\mathfrak{F}_{1}$, then we are done.  If not, by Lemma \ref{le:4.16}, $\widetilde{G}$ has at least one pendant vertex, say $x$. Let $y$ be the unique neighbor of $x$ in $G$.
 Let $\widetilde{H}=\widetilde{G}-\{x,y\}$. If  $y$ is on the odd cycle of $G$, then by Lemma \ref{le:4.4},
$r(\widetilde{H})= 2m(H)-2\kappa(H)+1= 2m(H)-2c(H)+1$, a contradiction to Lemma \ref{le:2.12}.
Hence $y$ is not on the odd cycle of $G$. Since $r(\widetilde{G})=2m(G)-2\kappa(G)+1$, by Lemma \ref{le:4.3}, $y$ is not on any even
cycle of $G$. By Lemma \ref{le:4.4}, $r(\widetilde{H})=2m(H)-2\kappa(H)+1$.
Note that $\kappa(G)= \kappa(H)= c(H)-1=c(G)-1$. By induction hypothesis, $\widetilde{H}$ satisfies condition (b). Thus $\widetilde{G}$ satisfies condition (b).  ~~~~~~~~~~~~~~~~~~~~~~~~~~~~~~~~~~~~~~~~~~~~~~~~~~~~~~~~~~~~~~~~~~~~~~~~~~~~\quad $\square$

Now, we give an example that satisfies the equality in Theorem \ref{th:1.4}. Let $\widetilde{G}$ be a  mixed graph with an odd cycle $\widetilde{C}_{3}$ having even signature and each other cycle is $\widetilde{C}_{4}$  with $\sigma(\widetilde{C}_{4})\equiv q~(\mathrm{mod}~ 4)$ shown in Fig. 3.
By Lemmas \ref{le:2.2}, \ref{le:2.5} and \ref{le:2.7}, $r(\widetilde{G})=r(\widetilde{G}-\{u,v\})+2=(4-2)\kappa(G)+3+2=2\kappa(G)+5$ and $m(G)=m(G-\{u,v\})+1=2\kappa(G)+1+1=2\kappa(G)+2$.  Thus $r(\widetilde{G})=2m(G)-2\kappa(G)+1$.

\section{Proof of Theorem \ref{th:1.5}}
\noindent\begin{lemma}\label{le:5.1}
Let $\widetilde{G}$ be a bipartite cycle-disjoint mixed graph of order $n$. If $\widetilde{G}$ is nonsingular, then there exists a perfect matching, for any even cycle $\widetilde{C}_{q}$ with $\sigma(\widetilde{C}_{q})\equiv q~(\mathrm{mod}~4)$, $E_{1}(C_{q}) \cap E_{2}(G)\neq \emptyset$.
\end{lemma}
\begin{proof} Since $\widetilde{G}$ is nonsingular, by Remark 1.8, $2m^{\ast}(G)=n$. Since $G$ is a bipartite graph, by Lemma \ref{le:2.13}, $m^{\ast}(G)=m(G)$. Then $2m(G)= n$. So $G$ has a perfect matching. Suppose that the assertion is false. Then there are a perfect matching and an even cycle $\widetilde{C}_{q}$ with $\sigma(\widetilde{C}_{q})\equiv q~(\mathrm{mod}~4)$ such that $E_{1}(C_{q}) \cap E_{2}(G)= \emptyset$. Let $f(\widetilde{G}, \lambda)=\sum_{i=0}^{n}a_{i}\lambda^{n-i}$ be characteristic polynomial of $\widetilde{G}$.
By Lemma \ref{le:2.9}, $$a_{i}=\sum_{\widetilde{B}} (-1)^{\frac{1}{2}\sigma(\widetilde{B})+\omega(B)}2^{\beta(B)},$$
where the sum runs over all basic subgraphs $\widetilde{B}$ of order $i$ in $\widetilde{G}$.

Let $m=m(G)$. Then $a_{2m}\neq 0$ as $\widetilde{G}$ is nonsingular and $G$ has a perfect matching. Since $E_{1}(C_{q}) \cap E_{2}(G)= \emptyset$, for a given basic subgraph $\widetilde{U}$ of order $2m$, three related basic subgraphs of order $2m$ (including $\widetilde{U}$) emerges at the same time. One of them, say $\widetilde{U}$, contains $\widetilde{C}_{q}$ as a component, and the other two are obtained by replacing $\widetilde{C}_{q}$ with two distinct perfect
matchings of $\widetilde{C}_{q}$. The contributions to $a_{2m}$ of the three basic subgraphs are the following three terms:
$$(-1)^{\frac{1}{2}\sigma(\widetilde{U})+\omega(U)}2^{\beta(U)},~ (-1)^{\frac{1}{2}[\sigma(\widetilde{U})-\sigma(\widetilde{C}_q)]+\omega(U)-1+\frac{q}{2}}2^{\beta(U)-1},~ (-1)^{\frac{1}{2}[\sigma(\widetilde{U})-\sigma(\widetilde{C}_q)]+\omega(U)-1+\frac{q}{2}}2^{\beta(U)-1}.$$
Since $q$ is even and $\sigma(\widetilde{C}_{q})\equiv q~(\mathrm{mod}~4)$, the sum of the above three terms equals zero. Thus $a_{2m}=0$, a contradiction.
\end{proof}

Let $\mathcal{B}_{2m(G)}$ be the set of basic subgraphs of order $2m(G)$ in a mixed graph $\widetilde{G}$.

\noindent\textbf{Proof of Theorem \ref{th:1.5}.} \textbf{Necessity:} Since $\widetilde{G}$ is nonsingular, by Remark 1.8, $m^{\ast}(G)=n/2$. So assertion (a) holds. By Lemma \ref{le:5.1},  assertion (b) holds.

 \textbf{Sufficiency:} We apply induction on the order $n$ of $\widetilde{G}$. If $n=2$, then the result holds trivially.
 Now assume that $n\geq 3$. For any perfect matching of $\widetilde{G}$, if for any even mixed cycle $\widetilde{C}_{q}$ of $\widetilde{G}$,
$E_{1}(C_{q})\cap E_{2}(G)\neq \emptyset$. Let $\alpha(G)$ be the number of perfect matchings of $G$. We claim that $\alpha(G)=1$. If $\alpha(G)\geq 2$, note that
$\widetilde{G}$ is a bipartite cycle-disjoint mixed graph with perfect matching, in this case, there
must exists an even mixed cycle $\widetilde{C}$ of $\widetilde{G}$ and a perfect matching, $E_{1}(C)\cap E_{2}(G)=\emptyset$, a contradiction. By Lemma \ref{le:3.2}, we have $r(\widetilde{G})\geq 2m(G)$. Since $G$ is bipartite, by Lemma \ref{le:2.13}, $m^{\ast}(G)=m(G)$. Then by condition (a) and Inequality (2), $r(\widetilde{G})=2m^{\ast}(G)=n$. So $\widetilde{G}$ is nonsingular.

Otherwise, if there exist an even mixed cycle $\widetilde{C}_{q}$ of $\widetilde{G}$ and a perfect matching of $\widetilde{G}$, $E_{1}(\widetilde{C}_{q})\cap E_{2}(G)=\emptyset$, then by condition (b), either $\sigma(\widetilde{C}_{q})+q\equiv 2~(\mathrm{mod}~4)$ or $\sigma(\widetilde{C}_{q})$ is odd. For any basic subgraph $\widetilde{B}$ of order $2m(G)$ of $\widetilde{G}$, $\widetilde{B}$ contains either $\widetilde{C}_{q}$ or a perfect matching of $\widetilde{C}_{q}$. Let $\widetilde{H}=\widetilde{G}-V(C_q)$.
 By induction hypothesis, $\widetilde{H}$ is nonsingular. So $a_{2m(\widetilde{H})}\neq 0$.   If $\sigma(\widetilde{C}_{q})+q\equiv 2~(\mathrm{mod}~4)$, then
 $$\begin{aligned}
a_{2m(\widetilde{G})}&=\sum_{\widetilde{B}\in \mathcal{B}_{2m(G)}} (-1)^{\frac{1}{2}\sigma(\widetilde{B})+\omega(B)}2^{\beta(B)}\\
&=\sum_{\widetilde{U}\in \mathcal{B}_{2m(H)}} (-1)^{\frac{1}{2}[\sigma(\widetilde{U})+\sigma(\widetilde{C}_q)]+\omega(U)+1}2^{\beta(U)+1}+2(-1)^{\frac{1}{2}\sigma(\widetilde{U})+\omega(U)+\frac{q}{2}}2^{\beta(U)}\\
&=\sum_{\widetilde{U}\in \mathcal{B}_{2m(H)}} (-1)^{\frac{1}{2}\sigma(\widetilde{U})+\omega(U)}2^{\beta(U)}[(-1)^{\frac{1}{2}\sigma(\widetilde{C}_q)+1}\cdot 2+2\cdot (-1)^{\frac{q}{2}}]\\
&=\pm 4a_{2m(\widetilde{H})}\neq 0.
\end{aligned}$$
If $\sigma(\widetilde{C}_{q})$ is odd, then
$$\begin{aligned}
a_{2m(\widetilde{G})}&=\sum_{\widetilde{B}\in \mathcal{B}_{2m(G)}} (-1)^{\frac{1}{2}\sigma(\widetilde{B})+\omega(B)}2^{\beta(B)}\\
&=\sum_{\widetilde{U}\in \mathcal{B}_{2m(H)}} 2(-1)^{\frac{1}{2}\sigma(\widetilde{U})+\omega(U)+\frac{q}{2}}2^{\beta(U)}\\
&=\sum_{\widetilde{U}\in \mathcal{B}_{2m(H)}} (-1)^{\frac{1}{2}\sigma(\widetilde{U})+\omega(U)}2^{\beta(U)}2\cdot (-1)^{\frac{q}{2}}\\
&=\pm 2a_{2m(\widetilde{H})}\neq 0.
\end{aligned}$$
Thus $\widetilde{G}$ is nonsingular.   ~~~~~~~~~~~~~~~~~~~~~~~~~~~~~~~~~~~~~~~~~~~~~~~~~~~~~~~~~~~~~~~~~~~~~~~~~~~~~~~~~~~~~~~~~~~~~~~~~~~~~~~~~~~~\quad $\square$

\begin{figure}[htbp]
\centering
 \includegraphics[scale=1.2]{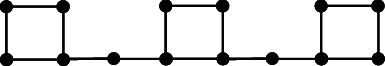}
\caption{An example that satisfies the equality in Theorem \ref{th:1.5}.}
\end{figure}

Now, we give an example that satisfies the conditions in Theorem \ref{th:1.5}. Let $\widetilde{G}$ be a  mixed graph shown in Fig. 4, where  $\widetilde{C}_4$ in the middle satisfies $\sigma(\widetilde{C}_{4})\equiv 0~(\mathrm{mod}~ 4)$ and the remaining two cycles $\widetilde{C}_4$ satisfy $\sigma(\widetilde{C}_{4})\equiv 2~(\mathrm{mod}~ 4)$. Then $\widetilde{G}$ satisfies (a) and (b) of Theorem \ref{th:1.5}.
By Lemmas \ref{le:2.2}, \ref{le:2.5} and \ref{le:2.9}, $r(\widetilde{G})=4+r(\widetilde{G}-V(C_4))=14=|V(G)|$. So $\widetilde{G}$ is nonsingular.

\section{Concluding remarks}

\quad ~In Theorems \ref{th:1.3} and \ref{th:1.4}, we characterize all mixed graphs $\widetilde{G}$~($\kappa(G)=c(G)-1$) with $H$-rank $r(\widetilde{G})=2m(G)-2\kappa(G)$ or $r(\widetilde{G})=2m(G)-2\kappa(G)+1$. When we remove the condition $\kappa(G)=c(G)-1$, these two theorems are incorrect. For example, let $\widetilde{G}$ be a mixed bicyclic graph shown in Fig. 5(a).
Then $\kappa(G)=c(G)-2$ and $\widetilde{G}$ does not satisfy (b) of Theorems \ref{th:1.3} and \ref{th:1.4}. However, $r(\widetilde{G})=2m(G)-2\kappa(G)$ when two cycles $\widetilde{C}_3$ have odd signature and $r(\widetilde{G})=2m(G)-2\kappa(G)+1$ when one cycle $\widetilde{C}_{3}$ has odd signature and the other one has even signature. Note that all mixed graph with $H$-rank $r(\widetilde{G})=2m(G)-2c(G)$ has been identified by Theorem \ref{th:1.1}. In addition, there is no mixed graph $\widetilde{G}$ with $H$-rank $r(\widetilde{G})=2m(G)-2c(G)+1$ by Lemma \ref{le:2.12}. Hence it is interesting to determine all  mixed graphs $\widetilde{G}$ with $H$-rank $r(\widetilde{G})=2m(G)-2\kappa(G)$ or $r(\widetilde{G})=2m(G)-2\kappa(G)+1$ when $0\leq \kappa(G)\leq c(G)-2$.

In Theorem \ref{th:1.5}, our analysis solely focuses on the non-singularity property of all bipartite cycle-disjoint mixed graphs. Nevertheless, conditions (a) and (b) are insufficient to ensure non-singularity for every cycle-disjoint mixed graph. As demonstrated in Fig. 5(b)~($\widetilde{C}_{3}$ on the left has even signature, $\widetilde{C}_{4}$ in the middle satisfies $\sigma(\widetilde{C}_{4})\equiv 0~(\mathrm{mod}~ 4)$ and $\widetilde{C}_{4}$ on the right satisfies $\sigma(\widetilde{C}_{4})\equiv 2~(\mathrm{mod}~ 4)$), the mixed graph $\widetilde{G}$ fulfills conditions (a) and (b); however, its $H$-rank satisfies $r(\widetilde{G})=n-1$. Therefore, the problem concerning the non-singularity of cycle-disjoint mixed graphs deserves further exploration.

\begin{figure}[htbp]
\centering
 \includegraphics[scale=1.2]{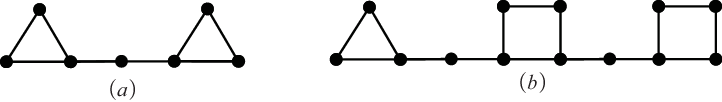}
\caption{Two counterexamples.}
\end{figure}

\medskip
\textbf{Declaration of competing interest}

The authors declare that they have no known competing financial interests or personal relationships that could have appeared to influence the work reported in this paper.

\medskip
\textbf{Date availability}

No date was used for the research described in the article.


\begin{thebibliography}{99}

%\bibitem{BBCR} F. Belardo, M. Brunetti, N.J. Coble, N. Reff, H. Skogman, Spectral of quaternion unit gain graphs, Linear Algebra Appl. 632 (2022) 15--49.
\bibitem{AP} P. Atkins, J. de Paula, Physical Chemistry, eighth ed., Oxford University Press, 2006.


\bibitem{BBD} J.H. Bevis, K.K. Blount, G.J. Davis, G.S. Domke, V.A. Miller, The rank of graph after vertex addition, Linear Algebra Appl. 265 (1997) 55--69.

\bibitem{BM} J.A. Bondy, U.S.R. Mutry, Graph Theory with Applications, Elsevier Science Publishing Co. Inc. 1976.

%\bibitem{ccz} S. Chang, A. Chang, Y.R. Zheng, The leaf-free graphs with nullity
%$2c(G)-1$, Discrete Appl. Math. 277(2020) 44--54.

%\bibitem{candl} S. Chang, J.X. Li, Graphs $G$ with nullity $n(G)-g(G)-1$, Linear Algebra
%Appl. 642 (2022) 251--263.

\bibitem{CHL} C. Chen, J. Huang, S.C. Li, On the relation between the $H$-rank of a mixed graph and the matching number of its underlying graph, Linear Multiliear Algebra. 66 (2018) 1853--1869.

\bibitem{CLZ} C. Chen, S.C. Li, M.J. Zhang, Relation between the $H$-rank of a mixed graph and the rank of its underlying graph, Discrete Math. 342 (2019) 1300--1309.

\bibitem{CG0}  Q.-Q. Chen, J.-M. Guo, The nullity of a graph with fractional matching number, Discrete Math. 345 (2022) 112919.

\bibitem{CG} Q.-Q. Chen, J.-M. Guo, The rank of a signed graph, Linear Algebra Appl. 651 (2022) 407--425.

\bibitem{CG1} Q.-Q. Chen, J.-M. Guo, Bounds of nullity for complex unit gain graphs, Linear Algebra Appl. 699 (2024) 569--585.

\bibitem{CL} B. Cheng, B.L. Liu, On the nullity of graphs, Electron. J. Linear Algebra. 16 (2007) 60--67.

%\bibitem{clt} B. Cheng, M.H. Liu, B.-S. Tam, On the nullity of a connected graph in
%terms of order and maximum degree, Linear Algebra Appl. 632 (2022) 192--232.

\bibitem{CS} L. Collatz, U. Sinogowitz, Spektren endlicher grafen, Abh. Math. Sem. Univ. Hamburg. 21 (1957) 63--77.

\bibitem{CIG} D. Cvetkovi\'{c}, I. Gutman, The algebraic multiplicity of the number zero in the spectrum of a bipartite graph, Matemati\'{c}ki Vesnik (Beograd) 9 (1972) 141--150.


\bibitem{CDS} D. Cvetkovik, M. Doob, H. Sachs, Spectra of Graphs: Theory and Applications, Academic, New York, 1980.

\bibitem{FHLL} Z.M. Feng, J. Huang, S.C. Li, X.B. Luo, Relationship between the rank and the matching number of a graph, Appl. Math. Comput.  354 (2019) 411--421.

\bibitem{HHL} S.J. He, R.X. Hao, H-J. Lai, Bounds for the matching number and cyclomatic number of a signed graph in terms of rank. Linear Algebra Appl. 572 (2019) 273--291.

\bibitem{GYY} J.-M. Guo, W.G. Yan, Y.-N. Yeh, On the nullity and matching number of unicyclic graphs, Linear Algebra Appl. 431 (2009) 1293--1301.

%\bibitem{DLZ} K.X. Du, Y. Lu, Q.N. Zhou, The gap between the rank of a complex unit gain graph and its underlying graph, Discrete Appl. Math. 357 (2024) 399--412.

%\bibitem{D} F. Duan, Characterizing the negative inertia index of connected graphs in terms of their girth, Discrete Math. 347 (2024) 113997.

%\bibitem{DY} F. Duan, Q. Yang, On graphs with girth $g$ and positive inertia index of $\frac{\lceil g\rceil}{2}-1$ and $\frac{\lceil g\rceil}{2}$, Linear Algebra Appl. 683 (2024) 98--110.

\bibitem{HHLG} S.J. He, R.X. Hao, H.-J. Lai, Q.Z. Geng, No mixed graph with the nullity $\eta(\widetilde{G})=|V(G)|-2m(G)+2c(G)-1$, arxiv:2304.06239v1.


\bibitem{K} S. Khan, Relation between the $H$-rank of a mixed graph and the girth of its underlying graph, Discrete Appl. Math. 373 (2025) 239--248.

\bibitem{LZX} S.C. Li, S.Q. Zhang, B.G. Xu, The relation between the $H$-rank of a mixed graph and the independence number of its underlying graph, Linear  Multilinear Algebra. 67 (2019) 2230--2245.

\bibitem{LGUO} X. Li, J.-M. Guo, No graph with nullity $\eta(G)=|V(G)|-2m(G)+2c(G)-1$, Discrete Appl. Math. 268 (2019) 130--136.

%\bibitem{LUWH} L. Lu, J.F. Wang, Q.X. Huang, Complex unit gain graphs with exactly one positive eigenvalue,  Linear Algebra Appl. 608 (2021) 270--281.

\bibitem{LWX} Y. Lu, L.G. Wang, P. Xiao, Complex unit gain bicyclic graphs with rank 2,~3 or 4, Linear Algebra Appl. 523 (2017) 169--186.


%\bibitem{LWZ1} Y. Lu, L.G. Wang, Q.N. Zhou, The rank of a complex unit gain graph in terms of the rank of its underlying graph,  J. Comb. Optim. 38 (2019) 570--588.

\bibitem{MF} X.B. Ma, X.W. Fang, An improved lower bound for the nullity of a graph in terms of matching number, Linear  Multilinear Algebra. 68 (2020) 1983--1989.


\bibitem{MWTDAM} X.B. Ma,  D.I. Wong, F.L. Tian, Nullity of a graph in terms of the dimension of cycle space and the number of pendant vertices, Discrete Appl. Math. 215 (2016) 171--176.

\bibitem{MWT} X.B. Ma, D.I. Wong, F.L. Tian, Skew-rank of an oriented graph in terms of matching number, Linear Algebra Appl. 495 (2016) 242--255.


\bibitem{M} B. Mohar, Hermitian adjacency spectrum and switching equivalence of mixed graphs, Linear Algebra Appl. 489 (2016) 324--340.


%\bibitem{REFF} N. Reff, Spectral properties of complex unit gain graphs, Linear Algebra Appl. 436 (2012) 3165--3176.


\bibitem{SU} E.R. Scheinerman, D.H. Ullman, Fractional Graph Theory: A Relational Approach to the Theory of Graphs, Wiley and Sons, New York, 1997.


\bibitem{SST} Y.Z. Song, X.Q. Song, B.-S. Tam, A characterization of graphs $G$ with nullity $|V(G)|-2m(G)+2c(G)$, Linear Algebra Appl. 465 (2015) 363--375.


\bibitem{WW} L. Wang, D.I. Wong, Bounds for the matching number, the edge chromatic numbber and the independence number of a graph in terms of rank, Discrete Appl. Math. 166 (2014) 276--281.

 \bibitem{WYL} Y. Wang, B.J. Yuan, S.D. Li, C.J. Wang, Mixed graphs with $H$-rank 3, Linear Algebra Appl. 524 (2017) 22--34.

\bibitem{WLM} W. Wei, S.C. Li, H.P. Ma, Bounds on the nullity, the $H$-rank and the Hermitian energy of a mixed graph, Linear  Multilinear Algebra. 69 (2021) 2469--2490.

%\bibitem{WMT} D. Wong, X.B. Ma, F.L. Tian, Relation between the skew-rank of an oriented graph and the rank of its underlying graph, European J. Combin. 54 (2016) 76--86.

\bibitem{YWY} J.L. Yang, L.G. Wang, X.W. Yang, Some mixed graphs with $H$-rank 4, 6 or 8,   J. Comb. Optim. 41 (2021) 678--693.

\bibitem{ZWS} Q. Zhou, D.I. Wong, D.Q. Sun, An upper bound of the nullity of a graph in terms of order and maximum degree, Linear Algebra Appl. 555 (2018) 314--320.

\bibitem{zwt} Q. Zhou, D.I. Wong, B.-S. Tam, On connected graphs of order $n$ with girth $g$ and nullity $n-g$, Linear Algebra Appl. 630 (2021) 56--68.


\bibitem{ZWT} Q. Zhou, D.I. Wong, F.L. Tian, Relation between the nullity of a graph and its matching number, Discrete Appl. Math. 313 (2022) 93--98.






















\end{thebibliography}
\end{document}